\documentclass[11pt,twoside]{article}
\usepackage{latexsym,amssymb,amsbsy}
\usepackage{pifont}
\usepackage{fourier}
\usepackage[T1]{fontenc}
\usepackage{a4wide}
\usepackage[svgnames]{pstricks}
\usepackage{pst-plot}
\usepackage[dvips]{graphicx}
\usepackage{theorem}
\usepackage{hyperref}
\usepackage{bookmark}
\usepackage{fancyhdr,caption}

\def\bm#1{\boldsymbol{#1}}

\hypersetup{colorlinks=true,linkcolor=HeadColour,citecolor=HeadColour}
\DeclareMathAlphabet{\mathcal}{OMS}{cmsy}{m}{n}

\title{Bifurcations of relative equilibria near zero momentum in Hamiltonian systems with spherical symmetry}

\author{James Montaldi  \\[12pt] 
\normalsize\color{Grey}School of Mathematics,
University of Manchester, \\
\normalsize\color{Grey} Manchester M13 9PL, 
UK}

\date{\color{Grey}February, 2014}

  {\smallskip\noindent{\color{magenta}\vrule width 1pt\kern2pt\vrule width 1pt}\kern1cm
   \begin{minipage}{0.8\textwidth}{\bfseries\emph{Question: }}}%
  {\end{minipage}\kern1cm{\color{magenta}\vrule width 1pt}}
\makeatletter
 \@addtoreset{equation}{section}
 \@addtoreset{figure}{section}
\makeatother
 \newgray{paleGray}{0.85}
 \newrgbcolor{Lyapunov}{0.9 0 0}
 \newrgbcolor{Elliptic}{0 0.8 0}
 \newrgbcolor{Unstable}{0.6 0.5 0}

 \newrgbcolor{HeadColour}{0.8 0 0}

\theoremheaderfont{\sffamily\bfseries\upshape}
\newtheorem{theorem}{Theorem}[section]
\newtheorem{lemma}[theorem]{Lemma}
\newtheorem{proposition}[theorem]{Proposition}

\newenvironment{proof}%
        {\addvspace\baselineskip\noindent {\sc Proof:\ }}%
        {\hfill \ding{114} \par\addvspace\baselineskip}  %
\newenvironment{proofof}[1]%
        {\addvspace\baselineskip\noindent {\sc Proof of #1:\ }}%
        {\hfill \ding{114} \par\addvspace\baselineskip}  %

{\theorembodyfont{\normalfont}

\newtheorem{definition}[theorem]{Definition}
\newtheorem{remark}[theorem]{Remark}
\newtheorem{remarks}[theorem]{Remarks}
}

\def\defn#1{{\bfseries\itshape #1}}

\def\restr#1{\,\vrule height0.8ex width.4pt
           depth1.2ex\lower1.0ex\hbox{\scriptsize $\,#1$}}

\def\Ad{\mathop\mathrm{Ad}\nolimits}
\def\Coad{\mathop\mathrm{Coad}\nolimits}

\def\coad{\mathop\mathrm{coad}\nolimits}
\def\rk{\mathop\mathrm{rk}\nolimits}
\def\re{\textsc{re}}
\def\eps{\varepsilon}

\def\bma{\bm{\alpha}}

\newcommand\JJ{\mathbf{J}}
\newcommand\jj{\mathbf{j}}

\newcommand{\qq}{\mathbf{q}}

\newcommand{\uu}{\mathbf{u}}
\newcommand{\xx}{\mathbf{x}}

\newcommand\CC{\mathbb{C}}
\newcommand\II{\mathbb{I}}
\newcommand\RR{\mathbb{R}}
\newcommand\TT{\mathbb{T}}
\newcommand\ZZ{\mathbb{Z}}

\newcommand\kE{\mathcal{E}}
\newcommand\RE{\mathcal{R}}
\newcommand\kG{\mathcal{G}}
\newcommand\kH{\mathcal{H}}
\newcommand\kK{\mathcal{K}}
\newcommand\kO{\mathcal{O}}
\newcommand\kR{\mathcal{R}}
\newcommand\KV{\mathcal{K}_V}
\newcommand\kP{\mathcal{P}}
\newcommand\kS{\mathcal{S}}

\newcommand\maxid{\bm{\mathfrak{m}}}%\mbox{\textit{\bfseries m}}}

\renewcommand\d{\mathsf{d}}
\newcommand\ii{\mathsf{i}}
\newcommand\Mat{\mathsf{Mat}\,}
\newcommand\GL{\mathsf{GL}}
\newcommand\SO{\mathsf{SO}}
\newcommand\SU{\mathsf{SU}}
\renewcommand\gg{\mathfrak{g}}
\newcommand\gl{\mathfrak{gl}}

\newcommand\so{\mathfrak{so}}
\renewcommand\tt{\mathfrak{t}}

\def\Hbar{\overline{H}}
\def\half{{\textstyle \frac12}}

\begin{document}

\maketitle

\thispagestyle{empty}

\noindent\hrulefill % 

\smallskip

\centerline{\large\bf\color{gray} Abstract}

\medskip

\noindent 
For Hamiltonian systems with spherical symmetry there is a marked difference between zero and non-zero momentum values, and amongst all relative equilibria with zero momentum there is a marked difference between those of zero and those of non-zero angular velocity.  We use techniques from singularity theory to study the family of relative equilibria that arise as a symmetric Hamiltonian which has a group orbit of equilibria with zero momentum is perturbed so that the zero-momentum relative equilibrium are no longer equilibria.   We also analyze the stability of these perturbed relative equilibria, and consider an application to satellites controlled by means of rotors. 

\medskip

{\small

 }

\noindent\hrulefill % end of abstract

{\small 
\tableofcontents
}

%%%%%%%%%%%%%%%%%%%%%%%%%%%%
\section*{Introduction}
\label{sec:intro}

Analogous to the fact that in generic Hamiltonian systems equilibrium points form a set of isolated points, in generic Hamiltonian systems with symmetry, for each value of the momentum the relative equilibria are isolated.  It is therefore reasonable to parametrize the set of relative equilibria by the momentum value, at least locally.  As the momentum  value varies, one would then expect to see bifurcations occur, and many of these have similar descriptions to bifurcations occurring at equilibria in generic (non-symmetric) Hamiltonian systems, such as saddle-node, pitchfork and Hamiltonian-Hopf bifurcations (see \cite{BLM05} for a review).  However there is one class of transition that is due to the `geometry of reduction' and which occurs as a result of the momentum passing through a non-regular value in the dual of the Lie algebra, which for the group $\SO(3)$ means passing through 0. This type of geometric bifurcation was first investigated in \cite{Mo97} in the case where the angular velocity is non-zero even though the angular momentum vanishes.  This was extended in \cite{MR99} to zero angular velocity, where there is also an application to the dynamics of molecules.  

In this paper we describe these geometric transitions in more detail for the symmetry group $\SO(3)$.  The results also apply to other compact Lie groups, where the momentum value passes through a generic point of a reflection hyperplane in the Cartan subalgebra, but not to more degenerate points (see Remark\,\ref{rmk:other groups}). There are two cases to consider, the first is the `generic' one, where the velocity at the zero-momentum relative equilibrium is non-zero (a \emph{transverse relative equilibrium} in the terminology of Patrick and Roberts \cite{PR00}) and in this case the set of relative equilibria forms a smooth curve in the orbit space, as shown by Patrick \cite{Pa99}, and the curve can be naturally parametrized by the momentum.  The other case is where the relative equilibrium in question consists of equilibria.   Although non-generic in the universe of all symmetric Hamiltonian systems, this is the situation in systems governed by kinetic and potential energies.  In this case the set of relative equilibria generically forms three smooth curves in the orbit space, as is familiar from Euler's equations for the rigid body.

The question we address here is how the two are related: start with a relative equilibrium $p$ with zero momentum and zero velocity (that is, a zero-momentum equilibrium), and perturb the Hamiltonian so that the zero momentum relative equilibrium no longer has zero velocity.  How does the set of relative equilibria change?  We find in particular that in the class of all Hamiltonian systems with $\SO(3)$ symmetry, the zero-momentum equilibrium is of codimension 3: generically it would only be seen in 3-parameter families of such systems.  

The most familiar example of an $\SO(3)$-invariant system is the rigid body, with Euler's equations mentioned above, where the reduced picture  (in $\so(3)^*\simeq\RR^3$) of the set of relative equilibria consists of three lines through the origin corresponding to the three principal axes of the body, and it follows from results in \cite{MR99} that this persists when the rigid body motion is coupled to shape deformations. Now add terms with the effect that the zero momentum relative equilibrium is no longer an equilibrium. For most deformations, the three lines deform to three non-intersecting curves as shown in Figure\,\ref{fig:deformations}\,(vi), and the branches `reconnect' in different ways according to the specific deformation. This is analogous to how the two lines in the plane with equation $xy=0$ break up and reconnect to form the two branches of a hyperbola with equation $xy=\eps$, and which half-branch connects to which depends on the sign of $\eps$.

In the rigid body, it is well known that two of the branches are Lyapunov stable and one is unstable (even linearly unstable). When coupled with shape oscillations, one of the Lyapunov stable branches becomes linearly stable (elliptic) but not necessarily Lyapunov stable (this is provided the potential energy has a local minimum as a function of shape). The stability type can be followed in the deformation of the Hamiltonian, and we show where the transitions of stability occur in the deformations; one transition (between Lyapunov stable and elliptic) occurs at the point of zero momentum, and the others occur at points on the other branches.

The paper is organized as follows: in Section\,\ref{sec:reduction} we outline the approach we use for calculating relative equilibria based on the energy-Casimir method and the splitting lemma; it is the same method used in \cite{Mo97} and other papers since.  In Section\,\ref{sec:family of RE} we state Theorem\,\ref{thm:versal} which uses singularity theory to reduce the calculations of the geometry of the family of relative equilibria for a general family $\kH$ of Hamiltonians, to those of a particularly simple family $\kG$, and we find the relative equilibria for that family. 
In Section\,\ref{sec:stability} we study the stabilities of the bifurcating relative equilibria, and in Section\,\ref{sec:rotors} we consider an example of a rigid body (such as a satellite) equipped with three rotors, one parallel to each of the principal axes of inertia to find the family of relative equilibria when the rotors are given either fixed momenta or fixed speeds of rotation.

The paper concludes with Section\,\ref{sec:singularity theory} on singularity theory; this begins with a description of Damon's $\KV$-equivalence, which is the singularity theoretic equivalence required for the proof of Theorem\,\ref{thm:versal}, and then finishes with the proof of that theorem.

\paragraph{Acknowledgements}
I would like to thank J.E.~Marsden and P.S.~Krishnaprasad for suggesting the example of the system of a rigid body with rotors, which is discussed in Section \ref{sec:rotors}. This paper was completed during a stay at the Centre Interfacultaire Bernoulli (EPFL, Lausanne) and I would like to thank Tudor Ratiu and the staff of the centre for organizing such a productive environment, and for the financial support during my stay.  I would also like to thank the referees for their helpful suggestions.

%%%%%%%%%%%%%%%%%%%%%%%%%%%%%%%
%%%%%%%%%%%%%%%%%%%%%%%%%%%%%%%
\section{Reduction and slice coordinates}
\label{sec:reduction}

Let $(\kP,\omega)$ be a symplectic manifold with a Hamiltonian action of 
$\SO(3)$, which throughout we assume to be a free action.  The momentum map is denoted $\JJ:\kP\to\so(3)^*$, which without loss of generality can be assumed to be equivariant with respect to the coadjoint action on $\so(3)^*$, \cite{Mo97,OrtegaRatiu}. Since the action is free, $\JJ$ is a submersion.  Given an element $\xi\in\so(3)$ (the Lie algebra), we write $\xi_\kP$ for the associated vector field on $\kP$. Finally, let $H:\kP\to \RR$ be a smooth $\SO(3)$-invariant function, the Hamiltonian. 

Throughout this paper we assume we are given a relative equilibrium $p_e$ of this system, with $\JJ(p_e)=0$. That is, at $p_e$ there is an element $\xi\in\so(3)$ for which the Hamiltonian vector field at $p_e$ coincides with $\xi_P(p_e)$. This is equivalent to the  group orbit $\SO(3)\cdot p_e$ being invariant under the Hamiltonian dynamics. See for example \cite{Marsden92} or \cite{BLM05} for details.

Since we are interested in existence and bifurcations of relative equilibria near $p_e$, we describe the local normal form for Hamiltonian actions near such a point, and then we will use the normal form then on.  

Since $\JJ(p_e)=0$ one has by equivariance that $\JJ(g\cdot p_e)=0$ and hence $\d\JJ(\xi_\kP(p_e))=0$ (for all $\xi\in\so(3)$).  It follows that the tangent space to the group orbit $\so(3)\cdot p_e\subset\ker\d\JJ(p_e)$. Let $\kS$ be a slice to the group orbit $\SO(3)\cdot p_e$ at $p_e$ inside the submanifold $\JJ^{-1}(0)$.  It turns out (see for example \cite{GLS96}) that the pull-back of the symplectic form to $\kS$ is non-degenerate so that $\kS$ is symplectic (at least, in  a neighbourhood of $p_e$).  Since the action is free, the normal form of Marle-Guillemin-Sternberg states that there is an $\SO(3)$-invariant neighbourhood of $p_e$ which is $\SO(3)$-symplectomorphic to an invariant neighbourhood $U$ of the point $(e,0,0)$ in the symplectic space $Y$ with momentum map $\JJ_Y:Y\to\so(3)^*$ given by,
\begin{equation}\label{eq:MGS}
\begin{array}{rcl}
Y&=&\SO(3)\times \so(3)^* \times \kS,\\[6pt]
\JJ_Y(g,\,\rho,\,v) &=& \Coad_g\rho.
\end{array}
\end{equation}
The $\SO(3)$-action on $Y$ is simply $g'\cdot(g,\,\rho,\,v) = (g'g,\,\rho,\,v)$.  Since a neighbourhood of $p_e$ in $\kP$ is diffeomorphic to $U\subset Y$, the Hamiltonian $H$ on $\kP$ defines a Hamiltonian on $U$, which we also denote by $H$. This material is standard, and can be found for instance in the book of Ortega and Ratiu \cite{OrtegaRatiu}.  {Since all results of this paper are local, from now on we replace $\kP$ by the open set $U$ in $Y$, but then denote it $\kP$. }
Here $\Coad$ is the coadjoint action of $\SO(3)$, and is defined by
$\left<\Coad_g\mu,\,\xi\right> = \left<\mu,\,\Ad_{g^{-1}}\xi\right>$; note that  $\Coad_g$ is often written $\Ad_{g^{-1}}^*$. If we consider $\mu\in\so(3)^*\simeq\RR^3$ as a column vector, then $\Coad_g\mu=g \mu$, where the latter is just matrix multiplication.

In practice $\kS$ can often be interpreted as the phase space associated to `shape space', so corresponding to vibrational motions of the system, and $\SO(3)\times\so(3)^*$ as the phase space corresponding to rotational motions. The two types of motion are of course coupled.

We now proceed to pass to the quotient by the free group action, obtaining
\begin{equation} \label{eq:quotient}
\kP/\SO(3) \simeq \so(3)^* \times \kS
\end{equation}
and the \defn{orbit momentum map} is denoted $\jj:\so(3)^*\times\kS\to \RR$ and is independent of $s\in\kS$, just as the momentum map itself
is. For $\SO(3)$, $\jj(\mu,s) = \|\mu\|^2$ for a coadjoint-invariant norm on $\so(3)^*$; {when using coordinates we take $\|\mu\|^2=\half(x^2+y^2+z^2)$.} The reduced space $\kP_\mu\subset \kP/\SO(3)$ is then
$$\kP_\mu=\jj^{-1}(\|\mu\|^2) = \kO_\mu\times\kS,$$
where $\kO_\mu\subset \so(3)^*$ is the coadjoint orbit through $\mu$, which  is the 2-sphere containing $\mu$ if $\mu\neq0$ and degenerates to a point when $\mu=0$.

%%%%%%%%%%%
\paragraph{Energy-Casimir method} Now let $H$ be an $\SO(3)$-invariant smooth Hamiltonian on $\kP$. It descends to a smooth function on the orbit space $\Hbar:\kP/\SO(3)\to\RR$.  Write $H_\mu$ for the restriction of $\Hbar$ to the reduced space $\kP_\mu$; this is called the reduced Hamiltonian on $\kP_\mu$.  Since $\Hbar(\mu,s)=H(g,\mu,s)$ (which by hypothesis is independent of $g$), from now on we abuse notation and do not distinguish $H$ from $\Hbar$.

\def\Hbar{H}

Relative equilibria of the Hamiltonian system are solutions of the Lagrange multiplier problem $\d H - \xi\,\d\JJ$ on $\kP$; moreover the Lagrange multiplier $\xi\in(\so(3)^*)^*\simeq\so(3)$ can be interpreted as the angular velocity of the relative equilibrium. Equivalently, they are critical points of the reduced Hamiltonian $H_\mu$, for the appropriate value of $\mu$. See for example \cite{Marsden92} for details.

At points where $\mu\neq0$, $\jj$ is nonsingular so the critical points of $H_\mu$ are solutions of the \emph{reduced} Lagrange multiplier problem
\begin{equation} \label{eq:LagrangeMultiplier}
\d \Hbar - \lambda\, \d\jj=0
\end{equation}
for some $\lambda\in\RR$. Since $\jj(\mu,s)=\|\JJ(g,\mu,s)\|^2$ it follows that $\d\jj=2\,\JJ\cdot\d\JJ$ and comparing the two Lagrange multiplier equations one finds that $\xi$ and $\lambda$ are related by $\xi = 2\lambda\mu$, whenever $\mu\neq0$. Note that as $\mu\to0$ one may have $\lambda\to\infty$ so allowing $\xi\neq0$ with $\mu=0$.

On the other hand, at points where $\mu=0$, the restriction $H_0$ of $\Hbar$ to $\kP_0 = \{0\}\times\kS$ has a critical point wherever $\d_s \Hbar=0$.

\begin{definition}\label{def:non-degenerate}
A relative equilibrium at $\bar p\in\kP_\mu$ is said to be \emph{non-degenerate} if the hessian $\d^2H_\mu(\bar p)$ is non-degenerate.
This is equivalent to $\d_s^2H(p)$ being non-degenerate.
\end{definition}

We will be interested in the family of relative equilibria in a neighbourhood of a non-degenerate relative equilibrium with zero momentum.  Under this non-degeneracy assumption, it follows from the implicit function theorem  that in a neighbourhood of $\bar p_e=(0,s_e)$ in $\kP/\SO(3)$ we can solve the equation
\begin{equation}\label{eq:defn of s of mu}
\d_s\Hbar(\mu,s) = 0
\end{equation}
uniquely for $s=s(\mu)$. In other words $\d_s \Hbar(\mu,s(\mu)) \equiv 0$, and these are the only zeros of $\d_s \Hbar(\mu,s)$ in a neighbourhood of $(0,s_e)$.

Now define the function,
\begin{equation}\label{eq:defn of h}
\begin{array}{rcl}
h:\so(3)^* &\longrightarrow & \RR \\
  \mu\; &\longmapsto& \Hbar(\mu,s(\mu)).
\end{array}
\end{equation}
In fact this $h$ is only defined in a neighbourhood of the origin in $\so(3)^*$, but to save on notation we will ignore that here and continue to write $h:\so(3)^*\to\RR$.

\begin{proposition}\label{prop:reduced re}
Assume $\mu\neq0$. 
The following are equivalent:
\begin{enumerate}
\item $x=(\mu,s(\mu))\in\kP/\SO(3)$ is a relative equilibrium of the Hamiltonian system,
\item the map $(h,\,\jj):\so(3)^*\to\RR^2$ is singular at $\mu$,
\item $\d h(\mu) - \lambda\,\d\jj(\mu)=0$ for some $\lambda\in\RR$
\end{enumerate}
\end{proposition}

Note that at $\mu=0$,  (1) and (2) are equivalent, while (3) might not be. Furthermore, at any relative equilibrium, the differential $\d h(\mu) \in (\so(3)^*)^*\simeq \so(3)$, and can be identified with the velocity of the relative equilibrium (which is also an element of $\so(3)$). Details are in \cite{Mo97}.

\begin{proof}
At $x=(\mu,s(\mu))$ one has $\d \Hbar = (\d_\mu\Hbar,0) = (\d h,0)$, so (\ref{eq:LagrangeMultiplier}) is satisfied if and only if $\d h=\lambda \,\d\jj$, which is equivalent to $(h,\,\jj)$ being singular at $x$ since $\jj$ is non-singular when $\mu\neq0$.
\end{proof}

Applications of this approach can be found in \cite{MR99} (to relative equilibria of molecules) and in \cite{LMR01} (to relative equilibria of point vortices).

%%%%%%%%%%%
\paragraph{Parametrized version}
We will be interested in a parametrized family $H_u$ of $\SO(3)$-invariant Hamiltonians, with parameter $u\in U$, where $U$ is an open subset of $\RR^d$ for some $d$.
Write $\kH(z;u) = H_u(z)$ for such a family. We assume $\kH$ is a smooth $\SO(3)$-invariant function on $\kP\times U$, where $\SO(3)=\SO(3)$ acts trivially on $U$. Assume $0\in U$ and $H_0$ has a non-degenerate relative equilibrium at $p_e=(0,s_e)$.

The arguments of the previous paragraph can be extended to a parametrized family with no difficulty. For example the map $s$ obtained from solving (\ref{eq:defn of s of mu}) is now a map $s:\so(3)^*\times U\to \kS$, and one defines in the same way a smooth family of functions,
\begin{equation}\label{eq:defn of h family}
\begin{array}{rcl}
h:\so(3)^*\times U &\longrightarrow & \RR \\
  (\mu,\,u)\; &\longmapsto& \Hbar(\mu,s(\mu,u),u).
\end{array}
\end{equation}
And as in Proposition \ref{prop:reduced re}, the map $(h_u,\,\jj):\so(3)^*\to\RR^2$ is singular at $\mu\neq0$ if and only if the point $(\mu,\,s(\mu,\,u))$ is a relative equilibrium of $H_u$.

Notice that there are two types of parameter in this problem: firstly for a given Hamiltonian $H=H(g,\mu,s)$ (which is independent of $g$) there is a family of relative equilibria which is essentially parametrized by $\mu\in\so(3)^*$ (an \emph{internal} parameter), and secondly we consider a \emph{family} of such Hamiltonians, parametrized by an \emph{external} parameter $u\in U$.  The family of relative equilibria parametrized by $\mu$ (i.e.\ those of $H_\mu$) will then vary from one value of $u$ to another, and the aim of this paper is to study precisely how a particular type of family varies with an external parameter. The  particular type of family in question being one with a generic relative equilibrium with zero momentum and zero velocity deforming to one with non-zero velocity.  {It should perhaps be pointed out that there is a potential conflict of notation: $H_\mu$ is the reduced Hamiltonian (and $h_\mu$ is similar) while $H_u$ (and $h_u$) is the element of a family. Which is used at any time should be clear from the context (in fact, $H_\mu$ and $h_\mu$ are only used again in the proof of Theorem~\ref{thm:stability xi neq 0}), and in particular, the $H_0$ used in the next section refers to $H_u$ with $u=0$.}

\begin{remark}\label{rmk:other groups} 
The analysis in this paper is also valid for systems with symmetry group $\SU(2)$.  Moreover, it is also applicable to systems with compact symmetry group of rank greater than 1, provided the action on $\kP$ is free and the momentum value lies in a generic point of a wall of the Weyl chamber.  The argument is briefly as follows.   Let $G$ be a compact Lie group of rank $\ell$ say, and $p\in \kP$ is such that $\JJ(p)=\alpha\in\gg^*$ then one can find a symplectic cross section $R$ (see \cite{GLS96} or \cite{MT03}) which reduces the original system to a system invariant under $G_\alpha$. Let $Z_\alpha\lhd G_\alpha$ be its centre, which is a (submaximal) torus of dimension $r$ say. By reducing by $Z_\alpha$ one obtains a system invariant under $K:=G_\alpha/Z_\alpha$, with $\JJ_K(p)=0$ and we have $\rk(K)=\ell-r$. In particular, if $\alpha$ is a generic point of one of the reflection hyperplanes of the Weyl group action then $r=\ell-1$, so $K$ is a group of rank 1, which means it is isomorphic to either $\SO(3)$ or $\SU(2)$.

Finding a description of the geometry of the set of relative equilibria in an analogous family near points deeper in the walls of the Weyl chamber (points fixed by subgroups of the Weyl group larger than $\ZZ_2$) remains an open problem.  However, the set of relative equilibria near such a momentum value  but with generic (regular) velocity is described in \cite{Pa95} and \cite{Mo97} {(in the first of these it is shown that the set of nearby relative equilibria forms a smooth submanifold, and in the second that for each nearby momentum value, the number of relative equilibria with that momentum value is equal to the order of an appropriate Weyl group).}
\end{remark}

%%%%%%%%%%%%%%%%%%%%%%%%%%%%%%%
%%%%%%%%%%%%%%%%%%%%%%%%%%%%%%%
\section{The family of relative equilibria}
\label{sec:family of RE}

The main aim of this paper is to determine the behaviour of the family of relative equilibria in a neighbourhood of a non-transverse (relative) equilibrium when the Hamiltonian is deformed, so making it transverse in the sense of Patrick and Roberts \cite{PR00}.

For the organizing centre of our family, we consider an invariant Hamiltonian $H=H_0$ which has an \emph{equilibrium} $p_e$ with zero momentum value $\JJ(p_e)=0$, and we assume the equilibrium is non-degenerate in the sense of Definition \ref{def:non-degenerate}.  {Applying the procedure described after that definition, we have a function $h=h_0$ on $\so(3)^*$ whose linear part vanishes}. So after rotating the axes if necessary, the Taylor series of $h_0$ begins,
$$h_0(x,y,z) =  ax^2+by^2+cz^2 + O(3),$$
where $\mu=(x,y,z)$, and $O(3)$ represents terms of order 3. The genericity assumption we make throughout is that the coefficients $a,b,c$ are distinct, and we will assume $a>b>c$.

The set $\kR_0=\kR(H_0)$ of relative equilibria near $(0,0)\in\so(3)^*\times\kS$ coincides with the set of critical points of $(h_0,\,\jj):\so(3)^*\to \RR^2$. That is,
$$\kR_0 = \left\{\mu\in\so(3)^* \mid \rk F(\mu)\leq1\right\},$$
where $F(\mu)$ is the Jacobian matrix at $\mu$ of $(h_0,\,\jj)$, and $\rk F(\mu)$ is the rank of that matrix.

The aim now is to study the geometry of the set of relative equilibria of any deformation of $H_0$ and hence of $h_0$.  Consider any family of $\SO(3)$-invariant Hamiltonian systems containing $H_0$ as above, and let $\kH$ be the resulting  deformation of $H_0$ parametrized by $u\in U$, and for each $u\in U$ write $h_u$ for the corresponding reduced Hamiltonian on $\so(3)^*$ as constructed above---see (\ref{eq:defn of h}).

To study the family $\kR(\kH)$ of relative equilibria of any given family $\kH$, we describe an auxilliary `universal' 3-parameter family $\kG$ of functions $G_{\bma}$ on $\so(3)^*$, with $\bma\in\RR^3$, based on the function $G_0$ which is just the quadratic part of the given $h_0$.  Using singularity theory techniques we show that the family $\kR(\kH)$ of relative equilibria for the given deformation $\kH$ is the inverse image under a smooth map $\varphi$ of the family $\kR(\kG)$ of `relative equilibria' of $\kG$ (that is of critical points of $(G_{\bma},\jj)$), and then in Section~\ref{sec:study-univ-family} below we study the geometry of this universal family.

Thus, given $h_0(\mu)=ax^2+by^2+cz^2+O(3)$ as above, with $a>b>c$, we define
$$G_0(\mu) = ax^2+by^2+cz^2$$
where $\mu=(x,y,z)$, and a deformation
\begin{equation}\label{eq:G family}
\kG(x,y,z;\,\alpha,\beta,\gamma) = G_0(x,y,z) + \alpha x +\beta y+\gamma z\,;
\end{equation}
for brevity we write $\bma=(\alpha,\beta,\gamma)$ for the parameter, so we have $G_{\bma}(\mu) = \kG(\mu;\,\bma)$. Although the functions $G_{\bma}$ are an artefact of the problem, and not related directly to the dynamics of $H_u$, we do refer to their relative equilibria as if they did arise as Hamiltonians. Notice that the relative equilibria of the family $\kG$ correspond to critical points of $(G_{\bma},\,\jj)$, {which by Lagrange multiplier theory are the points of tangency of the ellipsoid/hyperboloid $G_{\bma}(\mu)=\mathrm{const}.$ with the sphere $\jj=\mathrm{const}$.}. 

\begin{theorem}\label{thm:versal}
Let $H_0$ be an $\SO(3)$-invariant Hamiltonian with zero linear part as above (in particular, $a,b,c$ distinct), and let $\kH$ be an $\SO(3)$-invariant deformation of $H_0$ with parameter space $U$.   With the construction above, defining the family\/ $\kG$ from $H_0$, there exists a neighbourhood $U'$ of\/ $0$ in $U$, and a smooth map
\begin{eqnarray*}
\Phi:\so(3)^*\times U'&\longrightarrow& \so(3)^*\times\RR^3\\
 (\mu,\,u) &\longmapsto & (\Phi_1(\mu,\,u),\,\varphi(u))
\end{eqnarray*}
with $\varphi(0)=0$, such that $(\mu,u)\in\kR(\kH)$ if and only if\/ $\Phi(\mu,u)\in\kR(\kG)$.

In particular, if for each $u\in U'$, we define $\Phi_u:\so(3)^*\to \so(3)^*$, by  $\Phi_u(\mu) := \Phi_1(\mu,u)$ then $\Phi_u$ is a diffeomorphism which identifies the set $\kR(H_u)$ with the set $\kR(G_{\varphi(u)})$.
\end{theorem}

In other words, the family $\kG$ provides in some sense a \emph{versal} deformation of $H_0$, and in particular the set $\kR(H_u)$ of relative equilibria of a perturbation $H_u$ of $H_0$ is diffeomorphic to the set of relative equilibria of $G_{\varphi(u)}$.  The precise sense of `versal' here is with respect to the $\kK_V$-equivalence of the map $F=\d(h,\,\jj)$, where $V$ is the set of $2\times3$ matrices of rank at most 1.  A description of this equivalence and the proof of the theorem are given in Section \ref{sec:singularity theory}.

The structure of the set $\kR(\kH)$ is therefore derived from that of $\kR(\kG)$, and the geometry of the latter is described in the remainder of this section.  The stabilities of the relative equilibria arising in the perturbed Hamiltonians $H_u$ are discussed in Section\,\ref{sec:stability}.

%%%%%%%%%%%%%%%%%%%%%%%%%%%%%%%%%%%%
\subsection{Study of the `universal' family \texorpdfstring{$\kG$}{G}} \label{sec:study-univ-family}
Let $\kG(x,y,z)$ be as given in (\ref{eq:G family}), with $a>b>c$ distinct, as in the theorem above. The relative equilibria $\kR(\kG)$ for this family occur at points where $(G_{\bma},\jj)$ is singular, by Proposition \ref{prop:reduced re}, where $\bma=(\alpha,\beta,\gamma)$. Then
$$\RE(\kG) = \left\{(\mu,\,\bm{\alpha})\in\RR^6\mid
\rk\,F_{\bma}(\mu)\leq1\right\}.$$
where $F_{\bma}$ is the Jacobian matrix of $(G_{\bma},\,\jj)$:
\begin{equation}
  \label{eq:F-alpha}
F_{\bma}(x,y,z) = \pmatrix{2ax+\alpha & 2by+\beta & 2cz+\gamma \cr x & y & z }.  
\end{equation}
Here we have taken $\jj(\mu)=\|\mu\|^2=\half(x^2+y^2+z^2)$. 
Thus $\kR(\kG)$  is given by the vanishing of the three minors of the matrix, so by the three equations
$$\begin{array}{lcl}
2(b-c)yz + \beta z - \gamma y &=& 0\\
2(c-a)zx + \gamma x -\alpha z &=& 0\\
2(a-b)xy + \alpha y -\beta x &=& 0.
\end{array}$$
If these three minors are denoted $A,B$ and $C$ respectively, then there is an algebraic relation, namely $xA + yB + zC =0$ (as there is between the minors of any matrix).

When $\bma=0$, the equations become $xy=yz=zx=0$, whose solutions form the three coordinate axes, see Figure\,\ref{fig:deformations}\,(i) (the colours in the figure refer to stability of different branches, which is discussed in Section\,\ref{sec:stability}).

For each $\bma=(\alpha,\,\beta,\,\gamma)\in\RR^3$, we have the smooth map $F_{\bma}:\RR^3\to \Mat(2,3)$ (the space of $2\times3$ matrices). Notice that the origin in $\RR^6$ is the only point $(\xx,\bma)$ where $F(\xx,\bma)$ is a matrix of rank less than 1.  Let $V\subset\Mat(2,3)$ consist of those matrices of rank at most 1, and let $V^\circ$ denote its relative interior; that is, the set of matrices of rank equal to 1. So $V$ is the union of $V^\circ$ and the zero matrix. The set of relative equilibria for $G_{\bma}$ is 
$\kR_{\bma} = F_{\bma}^{-1}(V).$

Now, $V^\circ$ is a (locally closed) submanifold of $\Mat(2,3)$ of dimension 4 and codimension 2; indeed the Lie group $\GL(2)\times\GL(3)$ acts by change of bases on $\Mat(2,3)$ and has three orbits corresponding to the rank of the matrix: the origin, $V^\circ$ and $\Mat(2,3)\setminus V$.  If we can show that $F$ or $F_{\bma}$ is transverse to $V$, then it follows that away from the origin $\kR(\kG)$ is a submanifold of $\RR^6$, or respectively $\kR(G_{\bma})$  is a submanifold of $\so(3)^*\simeq\RR^3$, in both cases of codimension 2.

\begin{lemma}
\begin{enumerate}
\item The map $F:\RR^6\to \Mat(2,3)$ is an invertible linear map and hence transverse to $V$; consequently $\kR(\kG)$ is smooth (of dimension 4) except at the origin.
\item For $\bma=(\alpha,\beta,\gamma)\in\RR^3$, the map $F_{\bma}:\RR^3\to\Mat(2,3)$ is transverse to $V$ if and only if\/ $\alpha,\beta$ and $\gamma$ are all nonzero.
\end{enumerate}
\end{lemma}

This lemma tells us that the \emph{discriminant} $\Delta=\Delta(\kG)$ of the family $\kG$ is the subset of the unfolding space $\RR^3$ where $\alpha\beta\gamma=0$, which is the union of the three coordinate planes, see Figure\,\ref{fig:discriminant}, and for $\bma\not\in\Delta$ the set $\kR_{\bma}$ of relative equilibria is a smooth 1-dimensional submanifold of $\so(3)^*$; that is, it is a union of smooth curves.

\begin{figure} % Discriminant in unfolding space
\centering\psset{unit=0.9}
\begin{pspicture}(-3,-2.5)(3,3)
  %3 planes:
 \psline(-2,0)(-2,2)(2,2)(2,-2)(0,-2)
   \psline(-2,-0.5)(-2,-2)(-1,-2)
 \psline(0,-2)(-1,-2.5)(-1,1.5)(1,2.5)(1,2)
 \psline(-2,0)(-3,-0.5)(1,-0.5)(3,0.5)(2,0.5)
 \rput(1.7,-1.7){$\Delta_2$}
  %lines of intersection
 \psline{->}(0,0)(0,2.5) \rput(-0.2,2.6){$\gamma$} \psline(0,-0.5)(0,-2)
 \rput(0.3,1){$\Delta_1$}
 \psline{->}(0,0)(2.5,0) \rput(2.6,-0.2){$\beta$} \psline(-1,0)(-2,0)
 \psline{->}(0,0)(-1.5,-0.75) \rput(-1.4,-0.95){$\alpha$}
% \rput(2.3,1.8){$\Delta$}
\end{pspicture}
\caption{The discriminant $\Delta$ in parameter space is the union of three planes} \label{fig:discriminant}
\end{figure}
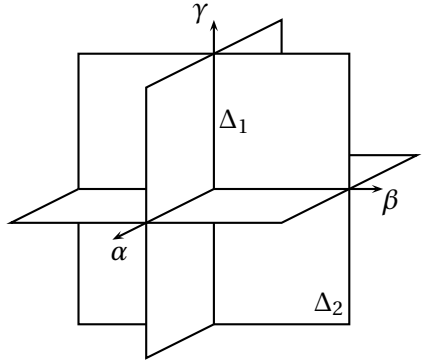

\begin{proof} (1) This is obvious from (\ref{eq:F-alpha}) above.

(2) Let $\hat\xx=(\hat x,\,\hat y,\, \hat z)\in T_{\mu}\RR^3\simeq\RR^3$. Then
$$\d F_{\bma}(\hat\xx) = \pmatrix{a\hat x & b\hat y & c\hat z \cr \hat x & \hat y& \hat z}.$$
Now, if $A\in V^\circ$ then $\rk A=1$ and the tangent space to $V$ at $A$ is
$$T_AV = \{B\in\Mat(2,3) \mid \qq BK=0\}$$
where $\qq\in\RR^2$ is a non-zero row-vector with $\qq A=0$, and $K$ is a $3\times2$ matrix whose image spans $\ker A$.
% ($\qq$ is unique up to scalar multiple and $K$ up to change of basis in $\RR^2$). 
To see this, let $A(t)$ be a smooth curve in $V$ with tangent vector $B$ at $A=A(0)$, and let $\qq(t)$ and $K(t)$ be the corresponding row-vector and matrix. Differentiating the condition $\qq(t)A(t)=0$ gives $\dot\qq A+\qq B=0$ and multiplying on the right by $K$ implies $\qq BK=0$; this shows that $T_AV$ is a subset of the $B$ with $\qq BK=0$, and a dimension count shows they are in fact equal.

Because $\d F_{\bma}$ is injective, in order to show $F_{\bma}$ is transverse to $V$ one needs to find two independent vectors $\hat\xx_1, \hat\xx_2$ such that $B_j := \d F_{\bma}(\hat\xx_j)\not\in T_AV,\;(j=1,2)$. The choice of the $\hat\xx_j$ depends on the point $\mu$ in question.
%For example, if $x\neq0$ then one can choose $\hat\mu_1=(y, -x,0)^T$ and $\hat\mu_2=(z,0,-x)^T$. Using $\qq=(-x,\,2ax+\alpha)$ and $K=\pmatrix{y&-x&0\cr z&0&-x}^T$.
A series of straightforward calculations in different cases ($x\neq0$, $\alpha\neq0$,\dots) show that indeed for $\bma\not\in\Delta$,  $F_{\bma}$ is transverse to $V$ and $F_{\bma}(\mu)\neq0$. If on the other hand $\alpha=0$ but $\beta\gamma\neq0$ then $F_{\bma}$ fails to be transverse to $V$ at the point
$$\mu=\left(0,\,\frac{\beta}{2(a-b)},\, \frac{\gamma}{2(a-c)}\right).$$
which is therefore the singular point of $F_{\bma}^{-1}(V)$---it is in fact a crossing of two components of the curve. A similar scenario occurs if $\beta=0,\,\alpha\gamma\neq0$ or if $\gamma=0,\,\alpha\beta\neq0$.

Furthermore, if $\alpha=\beta=0$ but $\gamma\neq0$ then there are two points where transversality fails:
\begin{equation}\label{eq:pitchfork points}
\mu=\left(0,\,0,\,\frac{\gamma}{2(b-c)}\right),\quad\mbox{and}\quad \mu=\left(0,\,0,\,\frac{\gamma}{2(a-c)}\right).
\end{equation}
Similarly if $\beta=\gamma=0$ or $\alpha=\gamma=0$ there are two singular points, given by analogous expressions.  In the case $\alpha=\gamma=0$ the two pitchfork points arise one on each side of the origin, at
\begin{equation}\label{eq:pitchfork points2}
\mu=\left(0,\,-\frac{\beta}{2(b-c)},\,0\right),\quad\mbox{and}\quad \mu=\left(0,\,\frac{\beta}{2(a-b)},\,0\right).
\end{equation}
as shown in Fig.~\ref{fig:deformations}\,(iv). Recall that we take $a>b>c$.
\end{proof}

With $\Delta$ the discriminant, let $\Delta_0 = \{0\}$, $\Delta_1$ be the set consisting of the points where precisely two of the planes intersect (ie, the union of the three axes without $\Delta_0$), and $\Delta_2$ the remaining points of $\Delta$. Then,
 $$\Delta = \Delta_0 \cup \Delta_1 \cup \Delta_2,$$
this being a disjoint union. $\Delta_1$ has 6 connected components, while $\Delta_2$ has 12. The geometry of the singular set of the deformation $(G_{\bma},\jj)$ (that is, the set of relative equilibria) depends on which stratum $\bma$ is in, as shown in the proof above. The descriptions are as follows (refer to Figure\,\ref{fig:deformations}).

\begin{figure}[t]
  \psset{unit=0.85,linewidth=1.5pt}
\centering

\begin{pspicture}(-2,-3)(3,2) % 3 lines through origin
 \psline[linecolor=Lyapunov](0,-1.7)(0,2)
 \psline[linecolor=Unstable](-2,0)(2,0)
 \psline[linecolor=Elliptic](-1.4,-0.7)(1.4,0.7)
 \psdot(0,0)
 \rput(0,-2.5){(i) $\bma = (0,0,0)$}
\end{pspicture}
\quad
\begin{pspicture}(-3,-3)(2,2) % on alpha axis: alpha>0
 \psline[linecolor=Lyapunov](-0.6,-2)(-0.6,1.6)
 \psline[linecolor=Elliptic](1.4,0.7)(0,0)
   \psline[linecolor=Lyapunov](0,0)(-0.6,-0.3)
   \psline[linecolor=Unstable](-0.6,-0.3)(-1.2,-0.6)
   \psline[linecolor=Elliptic](-1.2,-0.6)(-2,-1)
 \pscircle[fillcolor=white,linecolor=white,fillstyle=solid](-0.6,-0.6){0.1}%
 \psline[linecolor=Unstable](-3.2,-0.6)(0.8,-0.6)
 \psdot(0,0)
 \rput(-0.6,-2.5){(ii) $\bma=(\alpha,0,0)$}
\end{pspicture}
\quad
{\begin{pspicture}(-2,-3)(2,2.5) % on gamma axis: gamma=1
 \psline[linecolor=Lyapunov](0,0)(0,2)
 \psline[linecolor=Elliptic](0,0)(0,-0.5)
 \psline[linecolor=Unstable](0,-0.5)(0,-1)
 \psline[linecolor=Lyapunov](0,-1)(0,-2.2)
  \psline[linecolor=Unstable](-2,-1)(2,-1)
  \pscircle[fillcolor=white,linecolor=white,fillstyle=solid](-1,-1){0.3}
 \psline[linecolor=Elliptic](-1.6,-1.3)(1.4,0.2)
 \psdot(0,0)
 \rput(0,-2.5){(iii) $\bma=(0,0,\gamma)$}% $\alpha=\beta=0$}
\end{pspicture}}

\begin{pspicture}(-3,-4.5)(2,2) % on beta axis: beta=1
 \psline[linecolor=Unstable](-2,0)(-0.5,0)
   \psline[linecolor=Lyapunov](-0.5,0)(0,0)
   \psline[linecolor=Elliptic](0,0)(0.7,0)
   \psline[linecolor=Unstable](0.7,0)(2,0)
 \psline[linecolor=Lyapunov](-0.5,-1.7)(-0.5,2)
  \pscircle[fillcolor=white,linecolor=white,fillstyle=solid](-0.5,-0.6){0.15}%
 \psline[linecolor=Elliptic](-0.7,-0.7)(2.1,0.7)
 \psdot(0,0)
 \rput(-0.6,-3.5){(iv) $\bma=(0,\beta,0)$}
\end{pspicture}
{\begin{pspicture}(-4,-5)(3,2) % alpha=0 plane, with beta=0.1, gamma=1.5
 \psplot[linecolor=Lyapunov]{0.18}{0.7}{x 0.1 sub -1 exp x mul -1.5 mul} % lowest branch
 \psplot[linecolor=Unstable]{0.7}{2}{x 0.1 sub -1 exp x mul -1.5 mul} % lowest branch
 \psplot[linecolor=Lyapunov]{0}{0.05}{x 0.1 sub -1 exp x mul -1.5 mul}%highest branch
 \psplot[linecolor=Elliptic]{-0.1}{0}{x 0.1 sub -1 exp x mul -1.5 mul}%highest branch
 \psplot[linecolor=Unstable]{-2.2}{-0.1}{x 0.1 sub -1 exp x mul -1.5 mul}%highest branch
  \pscircle[linecolor=white,fillcolor=white,fillstyle=solid](-1.6,-1.4){0.2}
 \psline[linecolor=Elliptic](-2,-1.6)(1.3,-0.05) %line at  y = -beta = -0.1, z=-gamma/2 = -0.75%
 \psdot(0,0)
 \rput(0,-4){(v) $\alpha=0,\; \gamma>\beta>0$}
\end{pspicture}}
{\begin{pspicture}(-3,-5)(2,2) % alpha=0.1, beta=0.2,  gamma=0.5
 \parametricplot[linecolor=Elliptic]{-2.0}{-0.3}{t 0.1 add -1 exp t mul -0.2 mul t add t 2 mul 0.1 add -1 exp t mul -0.5 mul t 0.5 mul add}
 \parametricplot[linecolor=Unstable]{-0.3}{-0.11}{t 0.1 add -1 exp t mul -0.2 mul t add t 2 mul 0.1 add -1 exp t mul -0.5 mul t 0.5 mul add}
 \parametricplot[linecolor=Unstable]{-0.09}{-0.07}{t 0.1 add -1 exp t mul -0.2 mul t add t 2 mul 0.1 add -1 exp t mul -0.5 mul t 0.5 mul add}
 \parametricplot[linecolor=Lyapunov]{-0.07}{-0.055}{t 0.1 add -1 exp t mul -0.2 mul t add t 2 mul 0.1 add -1 exp t mul -0.5 mul t 0.5 mul add}
 \parametricplot[linecolor=Lyapunov]{-0.043}{0}{t 0.1 add -1 exp t mul -0.2 mul t add t 2 mul 0.1 add -1 exp t mul -0.5 mul t 0.5 mul add}
 \parametricplot[linecolor=Elliptic]{0}{2.0}{t 0.1 add -1 exp t mul -0.2 mul t add t 2 mul 0.1 add -1 exp t mul -0.5 mul t 0.5 mul add}
 \psdot(0,0)
 \rput(0,-4){(vi) $\gamma>\beta>\alpha>0$}
\end{pspicture}}
\caption{Deformations of $\kR_0$. Each diagram shows $\kR_{\bma}$ (up to diffeomorphism) for the corresponding values of $\bma$. The large dot on one of the curves in each diagram represents the origin $\mu=0$.
(i) corresponds to $\bma=0$, (ii), (iii) and (iv) to $\bma$ in three different components of $\Delta_1$, (v) to $\bma$ in one of the components of $\Delta_2$ and (vi) to $\bma\not\in\Delta$. The colours of the branches refer to their stability: red for Lyapunov stable, green for elliptic and brown for linearly unstable, see  Sec.\,\ref{sec:stability}.} \label{fig:deformations}
\end{figure}
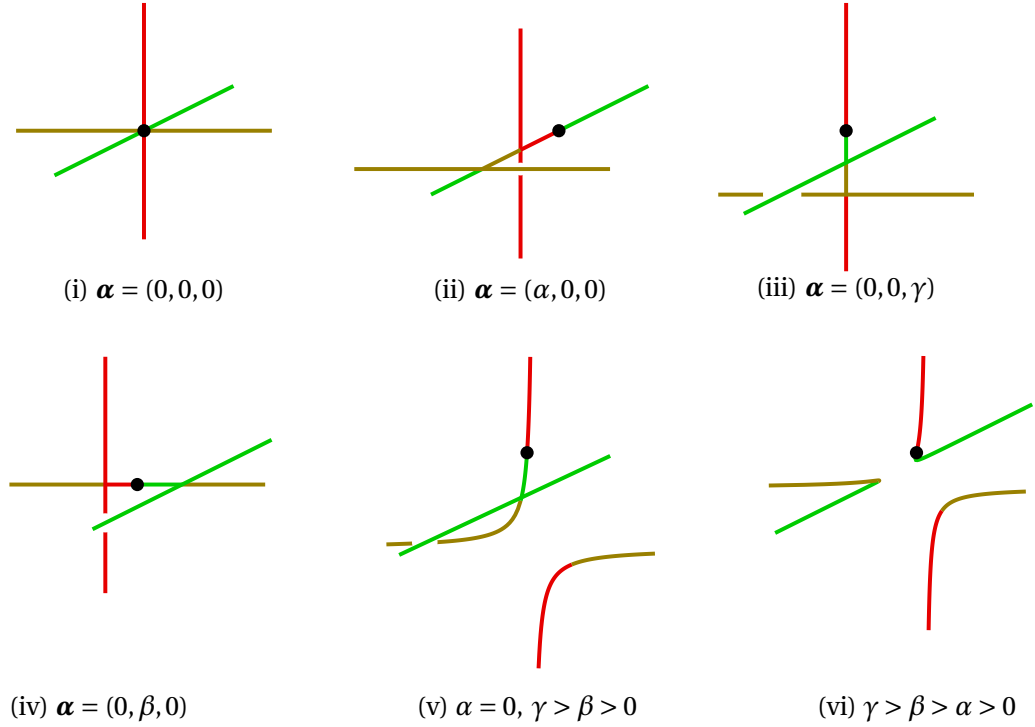

\begin{itemize}
\item For $\bma\not\in\Delta$ (so away from the discriminant: i.e.\ a generic deformation) the set $\RE_{\bma}$ of relative equilibria is formed of three smooth disjoint curves; see Figure\,\ref{fig:deformations}\,(vi). (The almost-corners in the figure are artefacts of the projection.) Before deformation when $\bma=0$, there are 3 lines through the origin or 6 `rays', and on deformation the rays are reconnected in such a way that opposite rays are no longer connected together.  There are 8 possible ways this can be done, corresponding to the 8 components in the complement of the discriminant (the 8 octants).  The origin must always lie on one of the 3 curves and is marked by the dot in the figure. For such a Hamiltonian, for small values of $\mu$ there are precisely two relative equilibria, both lying on the component passing through $0$, and as $\|\mu\|$ is increased there are two saddle-centre bifurcations, each creating a pair of relative equilibria. These bifurcations occur at the points closest to the origin on each of the other two components (i.e., where the sphere $\kO_\mu$ of the appropriate radius first touches the curve as $|\mu|$ increases from 0). See Fig.\,\ref{fig:bifurcations} for an illustration of the different types of bifurcation.

\item For $\bma\in\Delta_2$ ---a generic point of the discriminant--- two of the branches of $\RE_{\bma}$ intersect at a single singular point; see Figure\,\ref{fig:deformations}\,(v). The system determined by such a deformation undergoes a pitchfork bifurcation at this singular point and a saddle-centre bifurcation on the other branch as $\|\mu\|$ is increased from 0. 

\item For $\bma\in\Delta_1$ there are two crossings in $\RE_{\bma}$: the curve through $\mu=0$ meets both the other curves (at different points). There are therefore two pitchfork bifurcations in this system as $\|\mu\|$ is increased from 0. Since $a,b,c$ are distinct it follows from equation (\ref{eq:pitchfork points}) that the two bifurcations occur at different values of $\|\mu\|$. See Figures\,\ref{fig:deformations}\,(ii)--(iv).
\end{itemize}

Theorem \ref{thm:versal} shows that the family or relative equilibria for $\kH$ is diffeomorphic to that from $\kG$ (or more generally can be induced from it by a map $\varphi$ on parameters). However, this does not imply that where one has, for example, a sub-critical pitchfork so does the other---in fact it is not straightforward to match the stability types.  This is however proved in Theorem \ref{thm:stability xi=0}, so the conclusions above are relevant to a wider class of $\kH$ and not just to $\kG$. 

\begin{figure}[tp]
 \centering 
 \includegraphics[height=2cm,width=2cm]{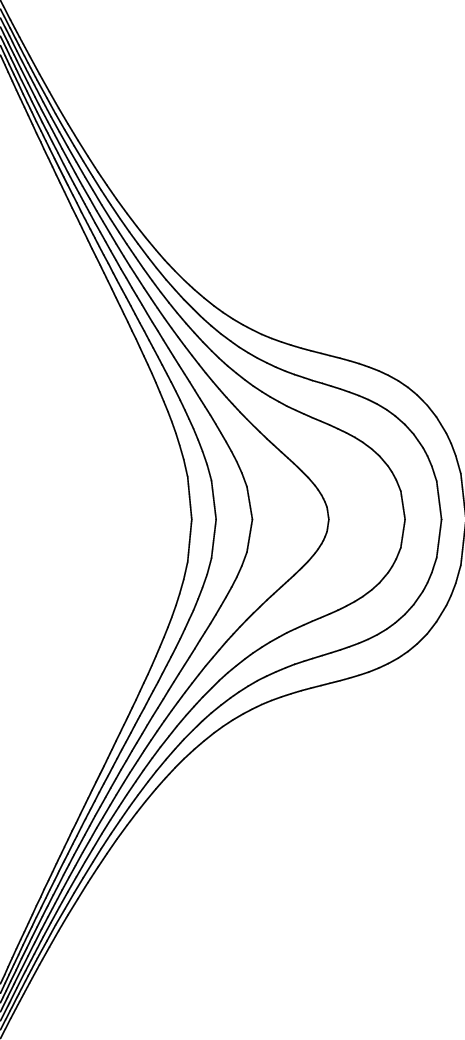}\qquad
 \includegraphics[height=2cm,width=2cm]{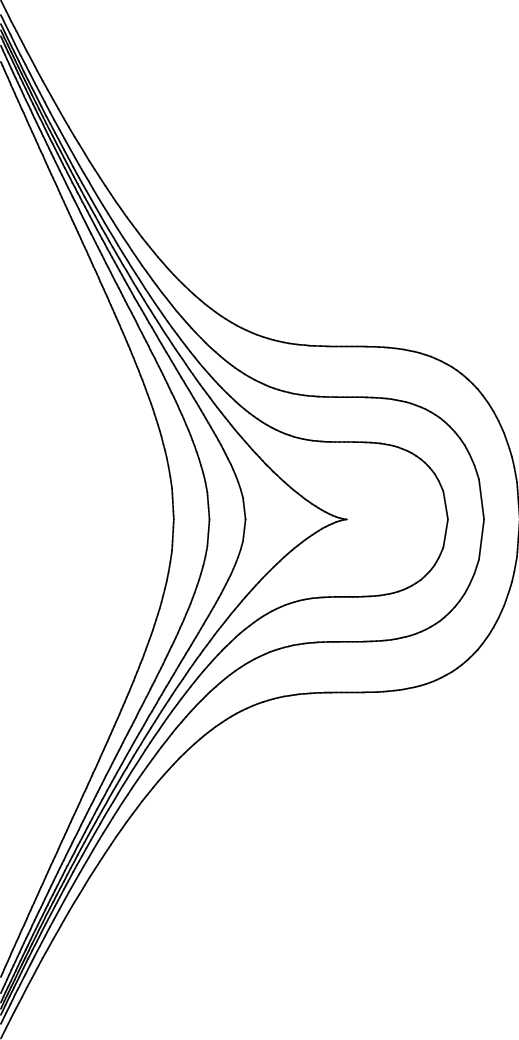}\qquad
 \includegraphics[height=2cm,width=2cm]{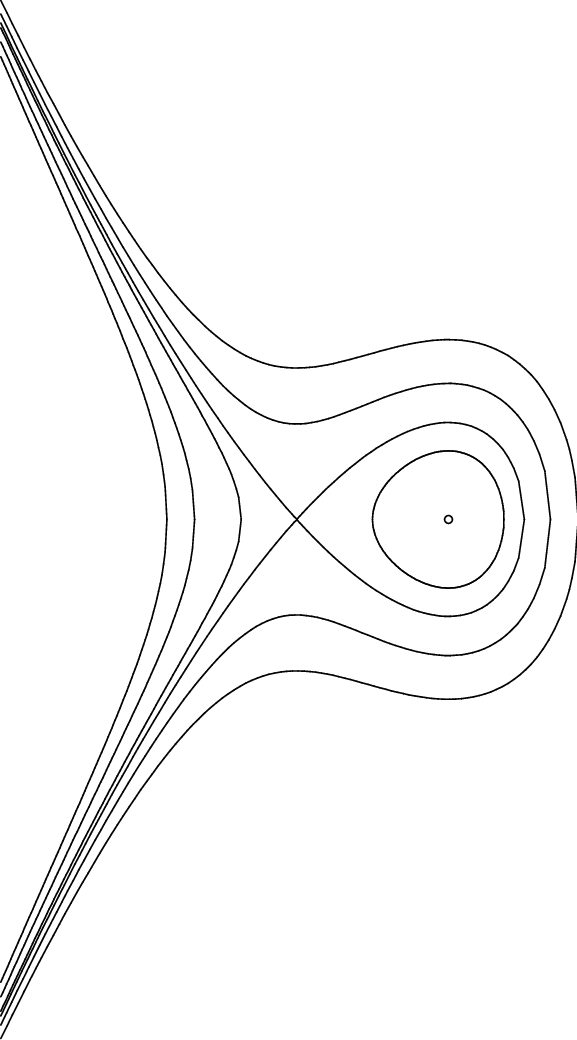}

Saddle-centre bifurcation

\vskip 5mm
 
%\begin{minipage}{0.49\textwidth}
 \includegraphics[height=2cm,width=2cm]{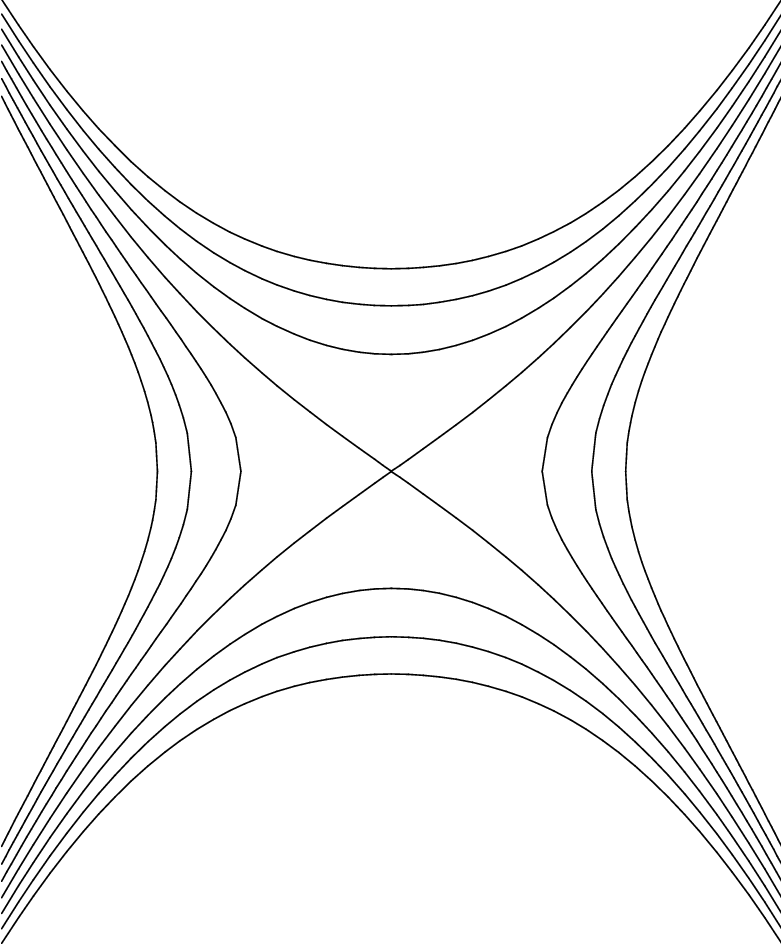}\qquad
 \includegraphics[height=2cm,width=2cm]{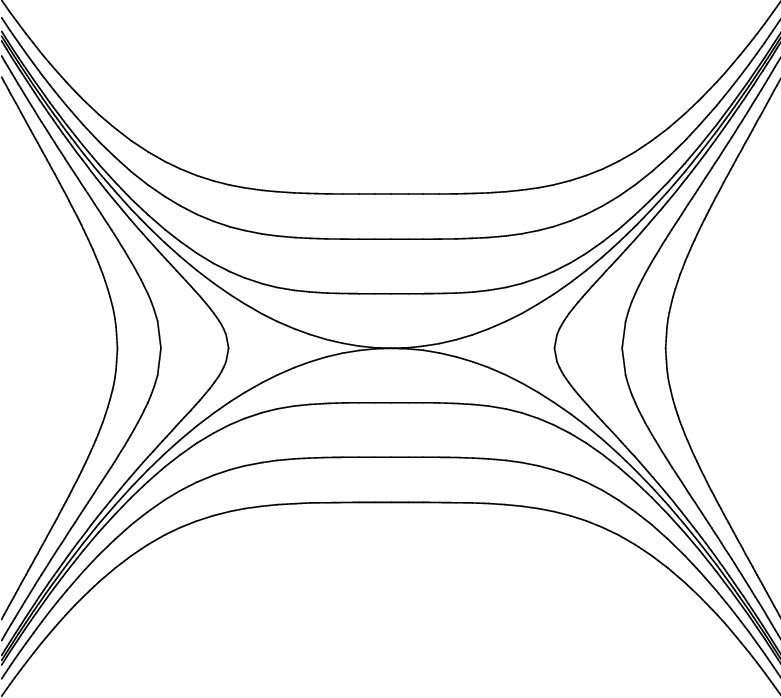}\qquad
 \includegraphics[height=2cm,width=2cm]{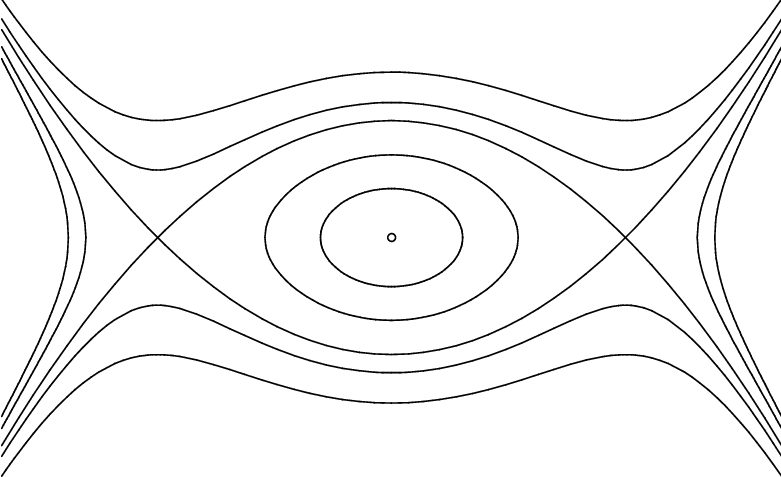}

 \centerline{Subcritical pitchfork  bifurcation}
%\end{minipage}\hfill

\vskip 5mm

%\begin{minipage}{0.49\textwidth}
 \includegraphics[height=2cm,width=2cm]{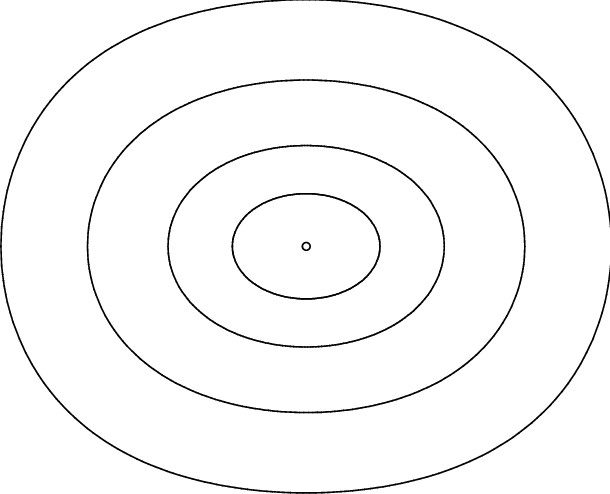}\qquad
 \includegraphics[height=2cm,width=2cm]{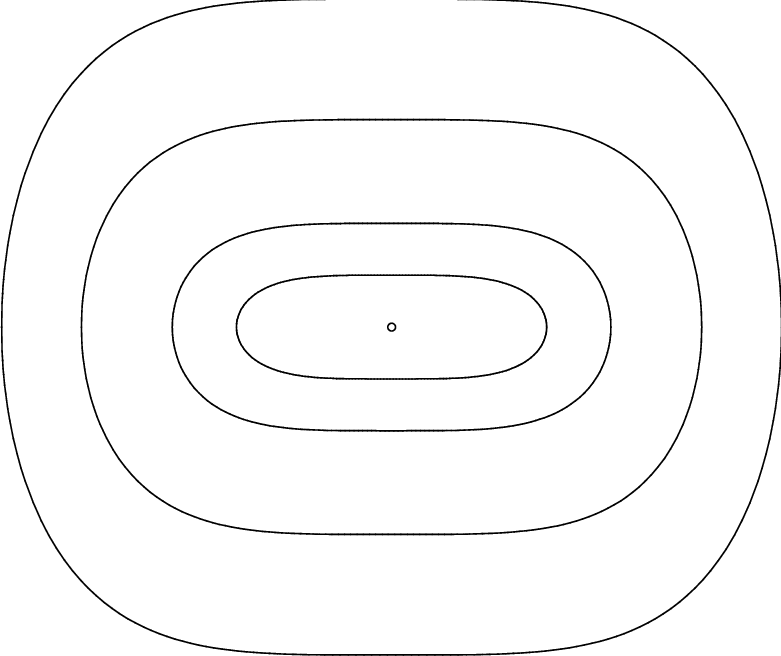}\qquad
 \includegraphics[height=2cm,width=2cm]{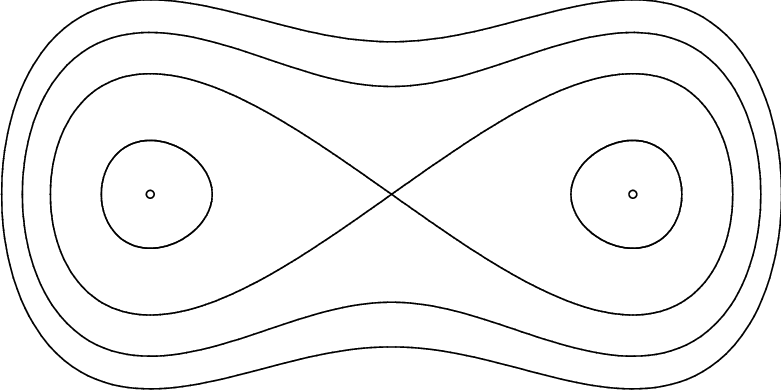}
 \centerline{Supercritical pitchfork bifurcation}
% \end{minipage}
\caption{Illustration of the three bifurcation types that occur}
\label{fig:bifurcations}
\end{figure}

%%%%%%%%%%%%%%%%%%%%%%%%%
\subsection{The energy-momentum discriminant of \texorpdfstring{$\kG$}{G}}

%%%% DISCRIMINANT with colour

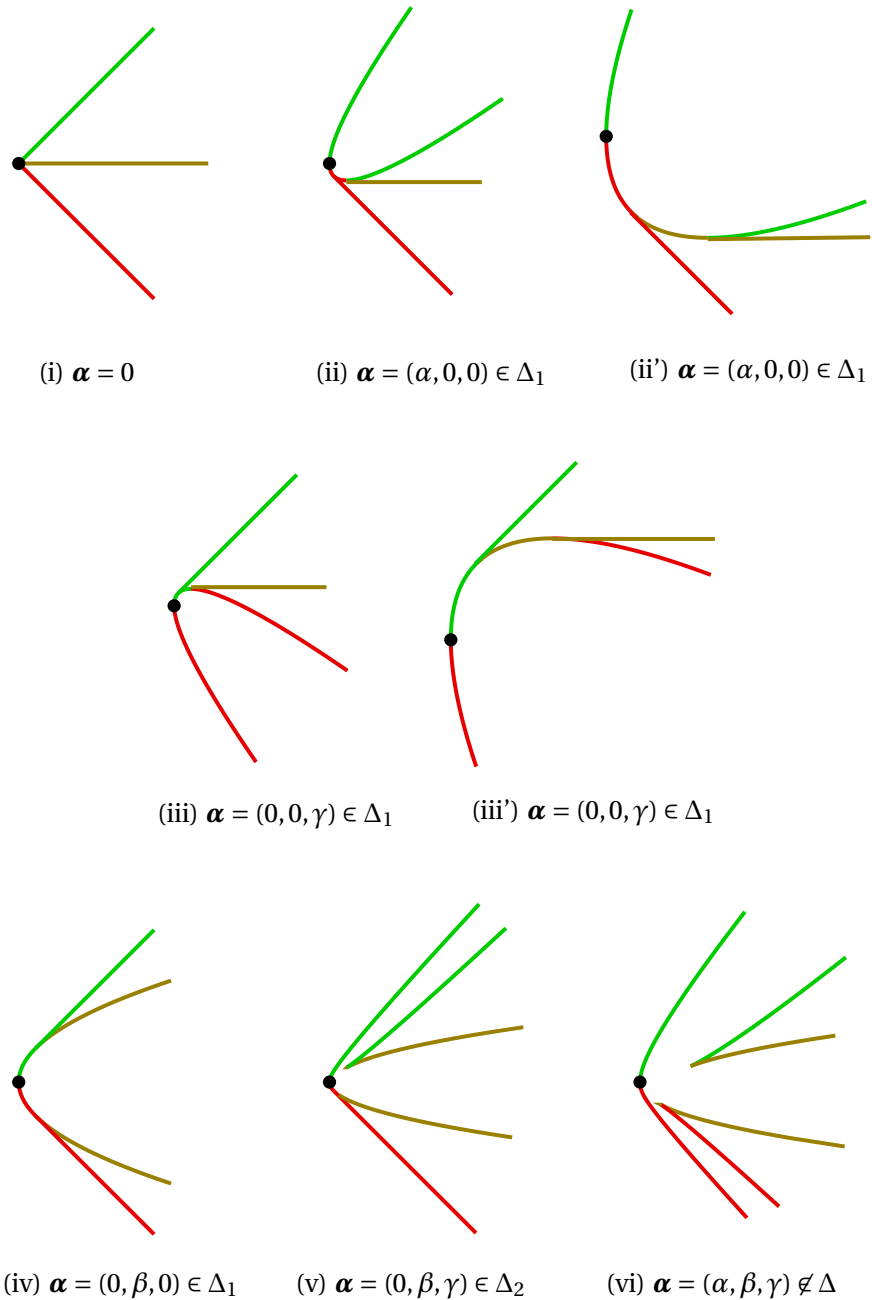
\begin{figure}[p]
\psset{linewidth=1.5pt,unit=0.9}
\centering
{\begin{pspicture}(-1,-4)(3.5,2)  % (0, 0, 0)
 \psline[linecolor=Elliptic](0,0)(2,2)
 \psline[linecolor=Unstable](0,0)(2.8,0)
 \psline[linecolor=Lyapunov](0,0)(2,-2)
 \rput(1,-3.1){(i) ${\bma}=0$}
 \psdot(0,0)
\end{pspicture}}
{\begin{pspicture}(-1,-4)(3.5,2) % (1, 0, 0)
\psset{unit=.5}
 \parametricplot[linecolor=Elliptic]{-2.2}{0}{t 2 exp 2 div t t 2 exp 2 div sub -1 mul}
 \parametricplot[linecolor=Lyapunov]{0}{1}{t 2 exp 2 div t t 2 exp 2 div sub -1 mul}
 \parametricplot[linecolor=Elliptic]{1}{3.2}{t 2 exp 2 div t t 2 exp 2 div sub -1 mul}
 \rput(0.125,-.375){\psline[linecolor=Lyapunov](0,0)(3.5,-3.5)}
 \rput(0.5,-0.55){\psline[linecolor=Unstable](0,0)(4,0)}
 \psdot(0,0)
 \rput(3,-6.2){(ii) ${\bma}=(\alpha,0,0)\in\Delta_1$}
\end{pspicture}}
{\begin{pspicture}(-0.5,-3.5)(4,3) % (1, 0, 0)
\psset{unit=3}
\rput(0,0.3){
 \parametricplot[linecolor=Elliptic]{-0.5}{0}{t 2 exp 2 div t t 2 exp 2 div sub -1 mul}
 \parametricplot[linecolor=Lyapunov]{0}{0.55}{t 2 exp 2 div t t 2 exp 2 div sub -1 mul}
 \parametricplot[linecolor=Unstable]{0.5}{1}{t 2 exp 2 div t t 2 exp 2 div sub -1 mul}
 \parametricplot[linecolor=Elliptic]{1}{1.6}{t 2 exp 2 div t t 2 exp 2 div sub -1 mul}
 \rput(0.12,-.375){\psline[linecolor=Lyapunov](0,0)(0.5,-0.5)}
 \rput(0.5,-0.507){\psline[linecolor=Unstable](0,0)(0.8,0.01)}
 \psdot(0,0)}
 \rput(0.7,-0.85){(ii') ${\bma}=(\alpha,0,0)\in\Delta_1$}
\end{pspicture}}

{\begin{pspicture}(-1,-4)(3.5,2) % (0, 0, 1)
\psset{unit=.5}
 \parametricplot[linecolor=Lyapunov]{-2.2}{0}{t 2 exp 2 div t t 2 exp 2 div sub}
 \parametricplot[linecolor=Elliptic]{0}{1}{t 2 exp 2 div t t 2 exp 2 div sub}
 \parametricplot[linecolor=Lyapunov]{1}{3.2}{t 2 exp 2 div t t 2 exp 2 div sub}
 \rput(0.125,.375){\psline[linecolor=Elliptic](0,0)(3.5,3.5)}
 \rput(0.5,0.55){\psline[linecolor=Unstable](0,0)(4,0)}
 \psdot(0,0)
 \rput(3,-6.2){(iii) ${\bma}=(0,0,\gamma)\in\Delta_1$}
\end{pspicture}}
{\begin{pspicture}(-0.5,-3.5)(4,3) % (0, 0, 1)
\psset{unit=3}
 \parametricplot[linecolor=Lyapunov]{-0.5}{0}{t 2 exp 2 div t t 2 exp 2 div sub}
 \parametricplot[linecolor=Elliptic]{0}{0.55}{t 2 exp 2 div t t 2 exp 2 div sub}
 \parametricplot[linecolor=Unstable]{0.5}{1}{t 2 exp 2 div t t 2 exp 2 div sub}
 \parametricplot[linecolor=Lyapunov]{1}{1.6}{t 2 exp 2 div t t 2 exp 2 div sub}
 \rput(0.12,.375){\psline[linecolor=Elliptic](0,0)(0.5,0.5)}
 \rput(0.5,0.507){\psline[linecolor=Unstable](0,-0.01)(0.8,-0.01)}
 \psdot(0,0)
 \rput(0.7,-0.85){(iii') ${\bma}=(0,0,\gamma)\in\Delta_1$}
\end{pspicture}}

{\begin{pspicture}(-1,-3.8)(3.5,3)  % (0, 1, 0)
\psset{unit=0.5}
 \parametricplot[linecolor=Unstable]{-3}{-1}{t 2 exp 2 div t}
 \parametricplot[linecolor=Lyapunov]{-1}{0}{t 2 exp 2 div t}
 \parametricplot[linecolor=Elliptic]{0}{1}{t 2 exp 2 div t}
 \parametricplot[linecolor=Unstable]{1}{3}{t 2 exp 2 div t}
 \rput(0.5,1){\psline[linecolor=Elliptic](0,0)(3.5,3.5)}
 \rput(0.5,-1){\psline[linecolor=Lyapunov](0,0)(3.5,-3.5)}
 \psdot(0,0)
 \rput(3,-6){(iv) ${\bma}=(0,\beta,0)\in\Delta_1$}
\end{pspicture}}
{\begin{pspicture}(-1,-3.8)(3.5,3) % \alpha=0
\psset{unit=3}
\def\alphaiszero{  2 t 2 exp mul 10 t mul 1 sub 2 exp div %
  t 2 exp 2 div add %%
  t 2 exp 2 div %
  2 t mul 5 10 t mul 1 sub mul div add %
  t 10 div sub}
 \rput(0.02125, -0.04375){\psline[linecolor=Lyapunov](0,0)(0.7,-0.7)}%%
 \parametricplot[plotpoints=200,linecolor=Elliptic]{-1.2}{0}{\alphaiszero}%
 \parametricplot[plotpoints=200,linecolor=Lyapunov]{0}{0.05}{\alphaiszero}%
 \parametricplot[plotpoints=200,linecolor=Unstable]{0.06}{0.087}{\alphaiszero}%
 \parametricplot[plotpoints=100,linecolor=Unstable]{.117}{0.3}{\alphaiszero}%
 \parametricplot[plotpoints=100,linecolor=Elliptic]{0.25}{1.3}{\alphaiszero}%
 \psdot(0,0)
  \rput(0.4,-1){(v) ${\bma}=(0,\beta,\gamma)\in\Delta_2$}
\end{pspicture}}
{\begin{pspicture}(-1,-3.8)(3.5,3) % all \neq 0
\psset{unit=3}
\def\generalcase{
 4.5 t 2 exp mul 20 t mul 1 sub 2 exp div %
 2 t 2 exp mul 10 t mul 1 sub 2 exp div add %
 0.5 t 2 exp mul add %%
 4.5 t 2 exp mul 20 t mul 1 sub 2 exp div %
 0.5 t 2 exp mul sub %
 0.9 t mul 20 t mul 1 sub div sub %
 0.4 t mul 10 t mul 1 sub div sub %
 0.1 t mul add}
 \parametricplot[plotpoints=100,linecolor=Lyapunov]{-1}{0}{\generalcase}
 \parametricplot[plotpoints=100,linecolor=Elliptic]{0}{0.0435}{\generalcase}
 \parametricplot[linecolor=Unstable]{0.1168}{0.3}{\generalcase}
 \parametricplot[linecolor=Lyapunov]{0.3}{1.1389}{\generalcase}
 \parametricplot[linecolor=Elliptic]{0.056}{0.07}{\generalcase}
 \parametricplot[linecolor=Unstable]{0.07}{0.087}{\generalcase}
 \psdot(0,0)
  \rput(0.4,-1){(vi) ${\bma}=(\alpha,\beta,\gamma)\not\in\Delta$}
\end{pspicture}}

\caption{Energy-momentum discriminants for the family $\kG$: the magnitude of the momentum increases to the right in each diagram.  The curves are the images of the sets of relative equilibria for the different values of $\bma$. (ii$^\prime$) is an expanded view of (ii), and similarly (iii') of (iii).  
\emph{Caveat}: see Remark \ref{rmk:double lines}\,(1). The colours refer to the stability of the relative equilibria, as in Figure\,\ref{fig:deformations}\,---\,see also \S\,\ref{sec:stability}.}
\label{fig:EM-discriminants}
\end{figure}

% 9/2 t^2/(20t-1)^2 = {4.5 t 2 exp mul 20 t mul 1 sub 2 exp div}
% 2t^2/(10t-1)^2 = {2 t 2 exp mul 10 t mul 1 sub 2 exp div}
% 1/2 t^2 = {0.5 t 2 exp mul}
% 9/10 z/(20z-1) = {0.9 t mul 20 t mul 1 sub div}
% 2/5 z/(10z-1) = {0.4 z mul 10 z mul 1 sub div}

The map $(h_{\bma},\jj)$ we are considering is the reduced energy-momentum map, and its singular set is the set $\kR_{\bma}$ of relative equilibria.  Its \emph{discriminant} is the image $(h_{\bma},\jj)(\kR_{\bma})\subset\RR^2$, and the fibres of the reduced energy-momentum map are diffeomorphic within each connected component of the complement of this discriminant.  It is therefore useful to know what this discriminant looks like.

It is important to be aware that the $\kK_V$-equivalence we use to reduce a general family to $\kG$ by a change of coordinates, does not respect the discriminant; that is, two maps which are equivalent in our sense do not necessarily have diffeomorphic discriminants. Bearing that in mind, the discriminants of the universal family $\kG$ are shown in Figure\,\ref{fig:EM-discriminants}, and a brief description of how they might appear for other families is given in Remark\,\ref{rmk:double lines}\,(1) below.

For the universal family $\kG$, the discriminant of $(G_0,\,\jj)$ consists of 3 rays as in Figure\,\ref{fig:EM-discriminants}\,(i). On a perturbation along the singular set  $\Delta_1$ of the unfolding discriminant it may look like (ii), (iii) or (iv) (depending on the component); for a generic point of the unfolding discriminant the perturbation may be of the form in (v), while a generic deformation is shown in (vi).

\begin{remarks}\label{rmk:double lines}
(1) In general, if $H_0$ has higher order terms,  the straight lines in Figure\,\ref{fig:EM-discriminants} may not be straight. Moreover, in the diagrams these lines are doubly covered by the energy-momentum map $(h_0,\jj)$, in that each point on one of the lines corresponds to 2 distinct relative equilibria. If $H_0$ is not an even function they would be expected to `open up' and become cusps of some order.  For example, comparing Figures\,\ref{fig:deformations}\,(i) and \ref{fig:EM-discriminants}\,(i) (both with $\bma=0$) each line in the first maps 2--1 to the corresponding ray in the second, and for a Hamiltonian which is not even, the rays in the latter figure would become cusps.

\smallskip

\noindent(2) It would be natural to attempt a classification of the reduced energy-momentum maps via \emph{left-right equivalence} ($\mathcal{A}$-equivalence), which consists of equivalence via diffeomorphisms in source and target. The diffeomorphism in the source would then relate the singular sets of the two maps and the one in the target would relate the discriminants.  However, the map $(\jj,G_0)$ is of infinite codimension with respect to this equivalence, because the map from the singular set $\RE$ to the discriminant is not 1--1 (it is in fact 2--1 away from the origin as pointed out above).  If we were to add appropriate cubic terms to $H_0$ or $G_0$, the  $\mathcal{A}$-codimension of the map would become finite (though considerably higher than 3) but calculations would be harder, and would also not be valid in settings where the higher order terms are absent, as in the example of \S\ref{sec:rotors}. 

\noindent(3) There has been no mention of what happens to the equilibrium in the family.  Since an equilibrium corresponds to a critical point of the Hamiltonian, and for $\bma=0$ the Hamiltonian (or the reduced function $h$) has a non-degenerate critical point at the origin, under any perturbation this non-degenerate critical point must persist.  Though it will no longer be at the origin, it will necessarily lie on one of the branches of the set of relative equilibria.  However, the $\KV$-equivalence we use does not respect critical points of $h$ (only critical points of $h$ relative to $\jj$) so we cannot use the universal family $\kG$ to determine its location.  The (unique) equilibrium will lie at some general point of one of the curves, and not, as might first be thought,  at a bifurcation point. Calculations suggest that which curve it is on depends on the signs of $a,b,c$, and that in the most physical case where all are positive, it lies on the branch that contains the point with zero momentum, and moreover it lies on the Lyapunov stable side of the zero momentum point.  But this has not been proved in general.
\end{remarks}

%%%%%%%%%%%%%%%%%%%%%%%%%%%%%%%
%%%%%%%%%%%%%%%%%%%%%%%%%%%%%%%
\section{Stability of the relative equilibria}
\label{sec:stability}

There are well-developed methods for proving the Lyapunov stability of relative equilibria,  based on Dirichlet's criterion for the stability of an equilibrium.  Since we are assuming the original action of $\SO(3)$ is free, the reduced spaces are smooth manifolds---indeed, we have seen they are of the form $\kO_\mu\times\kS$---and the Dirichlet criterion for \emph{reduced Lyapunov stability} (that is, Lyapunov stability on the reduced space) is that the Hessian of the reduced Hamiltonian $h_\mu$ should be positive or negative definite at the (relative) equilibrium in question.  Under this hypothesis, the relative equilibrium in the full phase space $\kP$ is $G$-Lyapunov stable \cite{Marsden92}, and it is shown by Lerman and Singer \cite{LS99}, based on work of Patrick \cite{Pa92}, that moreover it is Lyapunov stable relative to the possibly smaller group $G_\mu$, where $\mu$ is its momentum value, see also \cite{OrtegaRatiu}.  In \cite{Mo97,MT03} it is shown that if the Hessian is definite (in which case the relative equilibrium is \emph{extremal}) then on each nearby reduced space there is also an extremal (hence stable) relative equilibrium. 

If the momentum $\mu$ is non-zero, then the coadjoint orbits form a smooth foliation near $\mu$, and hence so do the reduced spaces near $(\mu,s)$.  Furthermore, if the reduced Hamiltonian has a non-degenerate critical point at $(\mu,s)$ then it does so for nearby reduced spaces as well (as observed essentially by Arnold \cite[Appendix 2]{MMCM}, at least in the case $\kS=0$), and if one Hessian is definite so are all nearby ones. 

On the other hand, if $p$ is a non-degenerate relative equilibrium with momentum $\mu=0$, then the local structure depends on the value of the velocity $\xi$.  If $\xi\neq0$ the situation was first studied by Patrick \cite{Pa95}, also by the author \cite{Mo97} and again in more detail by Patrick \cite{Pa99}, where he considers the eigenvalues of the linear approximations to the flow at such relative equilibria. We begin by considering the stability of nearby relative equilibria in this case, and afterwards we consider the case where both $\mu$ and $\xi$ are zero.

Returning to the decomposition (\ref{eq:quotient}), and the Hamiltonian $\Hbar(\mu,s)$, the Poisson Hamiltonian system on $\kP/G\simeq \so(3)^*\times\kS$ is,
\begin{equation}\label{eq:reduced equations}
\left\{\begin{array}{rcl}
  \dot\mu &=& -\coad_{D_\mu\Hbar}\mu,\\
  \dot s &=& J D_s\Hbar.
\end{array}\right.
\end{equation}
Here $J$ is the usual symplectic/Poisson structure matrix on $\kS$.  Linearizing these equations at the origin gives
$$\pmatrix{\dot\mu\cr \dot s} = L\pmatrix{\mu\cr s},$$
with
$$L = \pmatrix{-\coad_\xi& 0\cr C & J D^2_s\Hbar}$$
and where $\xi=D_\mu \Hbar(0,0)\in\so(3)$, and $C=J D^2_{s\mu}\Hbar(0,0)$ which is a linear map $\so(3)\to\kS$.  The spectrum of $-\coad_\xi$ is equal to $\bigl\{0,\pm\ii|\xi|\bigr\}$. More details can be found in \cite{Pa99,PR00}.

Write $L_0 = J D_s^2\Hbar(0,0)$. If the spectra of $L_0$ and of $-\coad_\xi$ are disjoint, then a change of coordinates (or choice of symplectic slice) can be chosen to eliminate the matrix $C$. If on the other hand, the spectra are not disjoint, one says there is a rotation-vibration resonance, and the matrix $C$ contributes a nilpotent term to the linear system.  See Remark \ref{rmk:stability}\,(c) below for further comments. 

Recall \cite{MHO} that an infinitesimally symplectic matrix is said to be,
\begin{itemize}
\item \emph{spectrally stable} if its spectrum is pure imaginary, 
\item \emph{linearly stable} or \emph{elliptic} if it is spectrally stable with zero niloptent part, 
\item \emph{strongly linearly stable} if it lies in the interior of the set of linearly stable matrices,
\item  \emph{linearly unstable} if it has an eigenvalue with non-zero real part. 
\end{itemize}
An equilibrium is said to be spectrally stable, elliptic, strongly linearly stable or linearly unstable if the linear part of the Hamiltonian vector field has the corresponding property. Note that if an equilibrium point is linearly unstable then it is also nonlinearly unstable.  In any continuous family of (relative) equilibria, no matter how it is parametrized, the transitions from one stability type to another occur only at points where the spectrum has a double eigenvalue. This could be at zero, in the transition between spectrally stable and unstable, and where the Hessian matrix of the Hamiltonian becomes degenerate, or a double imaginary eigenvalue (with mixed sign) resulting usually in a Hamiltonian-Hopf bifurcation and a change again from spectrally stable to unstable. The difference between linear stability and Lyapunov stability (in the full nonlinear system) lies with the `Krein sign' of the linear vector field, which is a question of whether the quadratic part of the Hamiltonian is positive definite or not---this is Dirichlet 's criterion for (Lyapunov) stability.  If the Hamiltonian is definite (positive or negative) then the equilibrium is strongly stable, and if a pair of complex conjugate eigenvalues crosses the origin but remains on the imaginary axis, then there is a transition from Lyapunov stable to elliptic. See \cite{BLM05} for more details.

%%%%%%%%%%%%%%%%%%%%
\subsection{Zero momentum, non-zero velocity}

We now present the first stability theorem appropriate for relative equilibria with zero momentum but non-zero velocity.

\begin{theorem}\label{thm:stability xi neq 0}
Suppose $G=\SO(3)$ acts freely on $\kP$ as before, with equivariant momentum map $\JJ$. Suppose that $p_0\in \kP_0$ (zero momentum) is a non-degenerate relative equilibrium  with non-zero angular velocity $\xi\in\so(3)$. Then,

\begin{enumerate}
\item there is a neighbourhood of $p_0$ in $\kP/G$ such that the relative equilibria in the neighbourhood form a smooth curve through $p_0$ intersecting each nearby reduced space in precisely 2 points;
\item if the Hessian of the reduced Hamiltonian  $\d^2\Hbar_0(p_0)$ is definite then on one side of $p_0$ on the curve the relative equilibrium will be Lyapunov stable, and if there is no rotation-vibration resonance then on the other it will be elliptic. 
\item  if $p_0$ is strongly linearly stable and there is no rotation-vibration resonance then throughout a neighbourhood of $p_0$ on the curve the equilibrium will be elliptic;
\item if $p_0$ is linearly unstable, then throughout the curve (in a neighbourhood of $p_0$), the relative equilibria will all be linearly unstable.
\end{enumerate}
\end{theorem}

\begin{remarks} \label{rmk:stability}
(a) The transition between Lyapunov stable and elliptic relative equilibria described in part (2) can be seen in Figure\,\ref{fig:deformations}\,(ii)--(vi), where the black dot represents the point $\mu=0$.\\
(b) In the case that $\kP_0$ is just a point (so $\kS=0$), the relative equilibria will be Lyapunov stable throughout the curve; see the example of the rigid body with rotors described in \S\ref{sec:rotors}.\\
(c) The rotation-vibration resonance was introduced in \cite{Pa99}.  If there is a rotation-vibration resonance and $C\neq0$ then it might be expected to see a singular Hamiltonian-Hopf bifurcation along the curve of relative equilibria (singular because it occurs at $\mu=0$, where the dimension of the reduced space changes). This possibility was also suggested at the very end of \cite{Mo97}, and to the author's knowledge has not yet been investigated as a bifurcation. An example of this phenomenon can be found in systems of point vortices---see \cite[\S 9]{LMR11}.
\end{remarks}

\begin{proof}
(1). Except for the intersection with reduced spaces, this is proved in \cite{Pa95}. Here we summarize the argument using the constructions described here in Section\,\ref{sec:reduction}. 
We define the function $h:\so(3)^*\to\RR$ as in (\ref{eq:defn of h}), and we have $\d h(0)=\xi\neq0$. Now $h_\mu$ is the restriction of $h$ to the sphere through $\mu$ and it follows that for small $\mu$ there are precisely two critical points of $h_\mu$, and as $\mu$ varies these form a smooth curve through $\mu=0$.  

\noindent(2) and (3). For each small, non-zero value of $\mu$, the function $h_\mu$ has an isolated minimum and an isolated maximum on the sphere $\kO_\mu$, and no other critical points. Suppose for (2) that $D_s^2\Hbar(p_0)$ is \emph{positive} definite (the argument for negative definite being similar), then the point which is a local minimum of $h_\mu$ will be Lyapunov stable, because at that point the reduced Hamiltonian on $P_\mu$ will also have a non-degenerate local minimum.

On the other hand, for the local maximum $\nu$ of $h_\mu$, and for both critical points in case (3), the function $H_\mu$ will have a critical point at $p=(\nu,s(\nu))$ for which the Hessian is indefinite. However, the spectrum of $L_\mu=L\restr{T_pP_\mu}$ will be a perturbation of the union of the spectra of $L_0$ and $\coad_\xi$ excluding the zero, and by hypothesis these are purely imaginary and disjoint. Moreover, in (2) the assumption that $\d^2H_0(p)$ is positive definite implies that $L_0$ is strongly linearly stable.  Thus in both (2) and (3) any sufficiently small perturbation of it has purely imaginary eigenvalues. It follows that the spectrum of $L_\mu$ will also be pure imaginary for $\mu$ sufficiently small.

\noindent (4). This follows a similar argument.  The spectrum of $L_\mu$ will contain perturbations of the spectrum of $L_0$, and as the latter has a non-zero real part, so will the former.
\end{proof}

%%%%%%%%%%%%%%%%%%%%
\subsection{Zero momentum, zero velocity}

We now turn to the unfolding of the relative equilibrium $p$ with $\mu=0$ and $\xi=0$, as in Theorem \ref{thm:versal}. Now, the equivalence relation used for the theorem only respects the set of relative equilibria; it does not respect dynamics, nor even the level sets of the energy-momentum map so one cannot \emph{a priori} deduce the stability of the bifurcating relative equilibria from studying the normal form $\kG$. 
On the other hand, stability only changes (along a branch of relative equilibria) through one of the two scenarios as described above.

Recall that the unfolding discriminant $\Delta$ of $\kG$ consists of three planes in $\RR^3$, that $\Delta_2$ denotes the open strata (regular points of $\Delta$) and $\Delta_1$ the points where two planes intersect; that is, $\Delta$ is the disjoint union of $\Delta_2$, $\Delta_1$ and the origin (see Figure\,\ref{fig:discriminant}). It follows from Theorem \ref{thm:versal} that the unfolding discriminant of any other family based on $H_0$ is pulled back from this $\Delta$ by some smooth map $\Phi$.  The following theorem justifies the pictures and stabilities shown in Fig.\,\ref{fig:deformations}.

\begin{theorem}\label{thm:stability xi=0}\parindent=0pt
Let $\kH$ be a family of\/ $\SO(3)$-invariant Hamiltonians, with parameter $u\in U$, and with $H_0$ having a non-degenerate relative equilibrium at $p_0\in\kP_0$ with $D_s^2H_0(p_0)$ positive definite, and with $D_\nu H_0(p_0) = 0$ (so the velocity is zero). Assume moreover that $D^2_\nu H_0(p_0)$ is a quadratic form whose three eigenvalues (with respect to an $\SO(3)$-invariant inner product on $\so(3)^*$) are distinct. (This is the setting of Theorem\,\ref{thm:versal}.)  Let $\varphi:U\to\RR^3$ be the map given by Theorem\,\ref{thm:versal}, inducing $\kH$ from $\kG$.

\begin{enumerate}
\item There is a neighbourhood of $p_0$ in $\kP/G$ in which the set of relative equilibria for $H_0$ consists of three curves; along one of these the relative equilibria are Lyapunov stable, along another they are elliptic and along the third they are linearly unstable.

\item There is a neighbourhood $M$ of the origin in $\so(3)^*$ such that, for 
$u\in U$ with $\varphi(u)\neq0$,

(i) if $\mu\in M$ is sufficiently small there are two relative equilibria as in Theorem \ref{thm:stability xi neq 0}, of which one is Lyapunov stable and the other elliptic;

For parts (ii)--(iv) we assume the Hamiltonian is analytic:

(ii) for $\varphi(u)\not\in\Delta$, as $\|\mu\|$ is increased, there are two saddle-centre bifurcations, one producing a Lyapunov stable relative equilibrium (\re) and an unstable \re, while the other produces an elliptic \re\ and an unstable \re;

(iii) for $\varphi(u)\in\Delta_2$ and $\mu\in M$, as $\|\mu\|$ is increased further, again one of the \re\ persists, and the other undergoes a supercritical\footnote{after a supercritical pitchfork bifurcation, a stable \re\ becomes two stable \re s and one unstable one, see Fig.\,\ref{fig:bifurcations}} pitchfork bifurcation; there is also a saddle-centre bifurcation. Whether it is the Lyapunov stable or elliptic \re\ that bifurcates depends on which connected component of\/ $\Delta_2$ contains $\varphi(u)$\,;

(iv) for $\varphi(u)\in\Delta_1$, as $\|\mu\|$ is increased, one \re\ persists, while the other undergoes two successive supercritical pitchfork bifurcations. Which persists and which bifurcates will depend on which component of\/ $\Delta_1$ contains $\varphi(u)$. 
\end{enumerate}
\end{theorem}

In Figures\,\ref{fig:deformations} and\,\ref{fig:EM-discriminants}, (i) corresponds to part (1) of the theorem, (ii)--(iv) to part 2(iv), (v) to part 2(iii) and (vi) to part 2(ii).  Note that the theorem does not imply that all these cases arise for every family: it is certainly possible that the image of $\varphi$ is contained in $\Delta_1$, for example.

\begin{proof}
(1) The existence of the three curves is part of the calculations in \S\,\ref{sec:study-univ-family}. 
And the stability part is proved in \cite[Section 2.5]{MR99}.

\noindent(2) 
(i) For $u\neq0$ we have $D_\nu H_u(0)\neq0$, consequently the relative equilibrium has non-zero velocity ($\xi\neq0$), and since $\xi$ is continuous function of $u$, it follows that for $u$ sufficiently small there is no rotation-vibration resonance, so this follows from Theorem \ref{thm:stability xi neq 0}.

\noindent(ii)
For each perturbation $H_u$, with $u\not\in\Delta$ the set $\kR$ consists of three disjoint curves, as described in \S\,\ref{sec:study-univ-family}.  By Lemma\,\ref{lemma:multiplicity} below, the restriction of the function $\jj$ to the curve $R_\alpha$ has at most 4 critical points in a fixed neighbourhood of the origin (for sufficiently small values of $\alpha$).  Now in this fixed neighbourhood of the origin, $\jj$ is increasing and on each curve reaches its maximum at the ends (for sufficiently small values of $\alpha$). The function therefore has at least one minimum on each branch, and in general an odd number of critical points (counting multiplicity) on each branch. The only way that is compatible with the upper bound of 4 is that there is a single non-degenerate critical point of $\jj$ on each branch, and so 3 in all.  Now one of the branches passes through the origin, where $\jj$ reaches its minimum value of $0$, while on the other two branches, the minimum will be a point where the branch is tangent to the momentum sphere (coadjoint orbit), so producing a saddle-centre bifurcation point.  These will be the two points of bifurcation mentioned in the theorem. 

There remains to show that the saddle-centre bifurcations involve the creation of critical points with the stated stability properties.  This is true for the model family $\kG$ by direct calculation.  Now consider a 1-parameter family of systems perturbing $\kG$ to $\kH$. By the multiplicity argument above, no other critical points are introduced, so the index of each critical point is the same for $\kG$ as the corresponding one for $\kH$, and the stability type depends only on the index.

\noindent(iii), (iv).  Here the proof is analogous to part (ii), but more straightforward as the pitchfork bifurcations will correspond to points where the map $\d(h,\jj)$ is not transverse to $V$, a property preserved by the equivalence we use in the proof above of Theorem\,\ref{thm:versal}, so is clearly preserved by the diffeomorphism.  If the pitchfork bifurcations were transcritical rather than sub- or super-critical then this would involve extra critical points which we know from the lemma below cannot happen.  The type of pitchfork and stability properties is the same in the family $\kH$ as for $\kG$ by the homotopy argument given above.

That there are no other bifurcations or loss of stability, except those involving an eigenvalue becoming zero, follows because there is no rotation-vibration resonance, so the spectra from the rotation part and the shape part are disjoint, and so can be continued with no extra multiple eigenvalues occurring. 
\end{proof}

\begin{remarks}
(i) If the hypothesis that $D^2_sH_0$ is positive definite is replaced by that of $L_0=JD^2_sH_0$ being strongly linearly stable, then the conclusions of the theorem are the same, but with Lyapunov stable replaced by elliptic throughout.  On the other hand, if $L_0$ is linearly unstable, then all the existence and bifurcation statements are the same except that all the \re\ are likewise linearly unstable.\\
(ii) In the proof of 2(ii) we need to assume the Hamiltonian is analytic in order to use methods of commutative algebra to estimate the number of critical points (see Lemma\,\ref{lemma:multiplicity} below); it would be surprising if this were an essential hypothesis. 
\end{remarks}

In the proof above we used the following lemma, which we prove using some commutative algebra based on ideas of Bruce and Roberts \cite{BR88}.

\begin{lemma} \label{lemma:multiplicity}
There is a neighbourhood $U_1$ of the origin in $\so(3)^*$ and a neighbourhood $U_2$ of\/ $0$ in $\RR^3$ such that for all $\alpha\in U_2$ the restriction of the function $\jj$ to $\kR_\alpha$ has at most 4 critical points in $U_1$, counting multiplicity.
\end{lemma}

\begin{proof}
Let $f$ be a smooth function on a manifold $M$.  The restriction of $f$ to a submanifold $X$ has a critical point at $x\in X$ if the graph of the differential 1-form $\d f$ intersects the conormal variety $N^*X$ at a point over $x$ (this is all in the cotangent bundle $T^*M$). The conormal variety is the bundle over $X$, given by
$$\left\{(x,\lambda)\in T^*M \mid x\in X,\; \lambda\in (T_xX)^\circ\right\}.$$
Here $(T_xX)^\circ$ is the annihilator of the tangent space $T_xX$. The total space of this bundle has the same dimension as the ambient manifold $M$.  The multiplicity of the critical point is equal to the intersection number of the graph and the conormal bundle. This multiplicity can be defined using modules of vector fields tangent to $X$, or using the sum of the ideals defining the graph of $\d f$ and the conormal bundle of $X$.

When $X$ is singular, the conormal bundle is replaced by the so-called logarithmic characteristic variety $LC^-(X)$ which is essentially the union over the (logarithmic) strata of $X$ of the closure of the conormal bundle to each stratum, see \cite{BR88} for details. Under certain algebraic conditions (namely $LC^-(X)$ should be Cohen-Macaulay), and providing everything is complex analytic, the multiplicity is preserved in deformations of the function, and without this algebraic condition the multiplicity is upper semicontinuous \cite[Proposition 5.11]{BR88}. 

We need to extend this by allowing the variety to deform as well as the function, but the semicontinuity is a general algebraic property, regardless of how the data deforms (provided the dimensions are constant). 

To return to our setting, first consider the central case with $\alpha=0$, and neglect the higher order terms in the Hamiltonian, so we are in the setting of \S\,\ref{sec:study-univ-family}, and consider everything complex. The variety $\kR_0$ consists of the three axes in $\CC^3$, and the function $\jj=\half(x^2+y^2+z^2)$. We are interested in critical points of $\jj$ restricted to $\kR_0$ (and later to $\kR_\alpha$).   
This clearly has a single critical point, namely the origin, so we need to understand its multiplicity. 

The variety $LC^-(\kR_0)\subset T^*\CC^3\simeq\CC^6$ consists of a 3-dimensional subspace for each of the axes, and a further one for the stratum $\{0\}$.  Explicitly, if we use the coordinates $(k,\ell,m)$ in the dual space (to form $T^*\CC^3$) then the union of the four 3-dimensional subspaces is,
$$LC^-(\kR_0)=\{x=y=m=0\}\cup\{y=z=k=0\}\cup\{z=x=\ell=0\}\cup\{x=y=z=0\}.$$
The ideal of this variety is $\left<xy,yz,zx,xk,y\ell,zm\right>$. Now the graph of $\d \jj$ is the set $\{k=x,\;\ell=y,\;m=z\}$, and a calculation shows that the algebraic intersection number is 4 (in fact the variety $LC^-(\kR_0)$ is Cohen-Macualay so the algebraic and geometric multiplicities coincide). That is, $\jj$ has a critical point of multiplicity 4 at the origin.

Now we wish to deform the set $\kR_0$ to $\kR_\alpha$ (which is smooth as $\alpha\not\in\Delta$), and take the new conormal variety but defined with the higher order terms of $H$ included and at the same time add in the higher order terms to $H$. 

Under this deformation, the multiplicity cannot increase (it remains constant if the whole family of conormal varieties is Cohn-Macaulay \cite{BR88}, but this turns out not to be the case here, explaining why the 4 ultimately drops to 3 in the course of the proof of Theorem\,\ref{thm:stability xi=0}). It follows that in the deformed setting there are at most 4 (complex) critical points, and therefore at most 4 real ones, as claimed. 
\end{proof}

%%%%%%%%%%%%%%%%%%%%%%%%%%%%%%%
%%%%%%%%%%%%%%%%%%%%%%%%%%%%%%%
\section{Example: rigid body with rotors}
\label{sec:rotors}

We give an application to the system consisting of a free rigid body with three freely rotating rotors attached so that their respective axes lie along the three principal axes of the body \cite{K85,BKMS92,MaSc93}.  The configuration space for this system is the Lie group $G=\SO(3)\times\TT^3$, where $\TT^3=S^1\times S^1\times S^1$ which acts by rotation of the three rotors.  A matrix $A\in \SO(3)$ corresponds to the attitude of the rigid body, while the components of  $\bm{\theta}=(\theta_1,\theta_2,\theta_3) \in\TT^3$ are the angles of rotation of the three rotors, relative to the body.

The Lagrangian of this system is given by the kinetic energy, which in a principal basis is
$$L=\half \bm{\omega}^T (\II-\II_r)\bm{\omega} + \half(\bm{\omega}+\dot{\bm{\theta}})^T\II_r(\bm{\omega}+\dot{\bm{\theta}}).$$
Here $\II_r$ is the diagonal matrix whose entries are the respective moments of inertia of the rotors about their axes, $\II$ is the inertia tensor of the rigid body with the rotors locked to the body and $\bm{\omega}\in \RR^3\simeq\so(3)$ is the angular velocity vector in the body; we assume $\II-\II_r$ is invertible. For details see Sec.~3 of \cite{BKMS92}, where $\bm{\omega}$ is denoted $\Omega$ and $\dot{\bm{\theta}}$ is denoted $\Omega_r$.

The corresponding momenta are therefore,
\begin{equation} \label{eq:rotors Legendre}
 \begin{array}{rclcl}
  \bm{\mu} &=& \partial L/\partial\bm{\omega} &=&  \II\,\bm{\omega} + \II_r\dot{\bm{\theta}}\,,\\
  \bm{\sigma} &=& \partial L/\partial\bm{\dot\theta} &=& \II_r(\bm{\omega}+\dot{\bm{\theta}}).
\end{array}
\end{equation}
Here $\bm{\mu}\in\so(3)^*$ is the angular momentum in the body, and $\bm{\sigma}\in\tt^*=\RR^3$ is the \emph{gyrostatic momentum} ($\bm{\mu}$ is denoted $m$ and $\bm{\sigma}$ is denoted $\ell$ in \cite{BKMS92}). The Hamiltonian is then
\begin{equation}\label{eq:rotor hamiltonian}
H = \half(\bm{\mu}-\bm{\sigma})^T(\II-\II_r)^{-1}(\bm{\mu}-\bm{\sigma}) + \half \bm{\sigma}^T\II_r^{-1}\bm{\sigma}
\end{equation}
The momentum map for the $G$-action is
$\JJ(A,\bm{\theta},\bm{\mu},\bm{\sigma}) = (A\bm{\mu},\,\bm{\sigma})$.
Indeed, $A\bm{\mu}$ is the angular momentum of the body in space, and $\bm{\sigma}$ is the conserved quantity due to the $\TT^3$-symmetry of the system, the gyrostatic momentum.

\subsection{The free system}
As a first step to analyzing the system as it is (with no external constraints), we reduce by the free $\TT^3$-action putting $\bm{\sigma}$ constant. This gives the reduced Hamiltonian on $T^*\SO(3)$,
$$H_{\bm{\sigma}}(A,\bm{\mu}) = \half(\bm{\mu}-\bm{\sigma})^T(\II-\II_r)^{-1}(\bm{\mu}-\bm{\sigma}) $$
(the other term in (\ref{eq:rotor hamiltonian}) is now a constant so can be ignored).  When $\bm{\sigma}=0$, $H_0$ is the usual rigid body Hamiltonian $\half \bm{\mu}^T(\II-\II_r)^{-1} \bm{\mu}$ which is homogeneous of degree 2, as is $G_0$ ---see prior to Theorem \ref{thm:versal}. 

Varying $\bm{\sigma}$ gives a straightforward example of the family $\kG$. Indeed if $(\II-\II_r)^{-1}=\mathrm{diag}[a,b,c]$ and $\bm{\sigma}=-(\alpha,\beta,\gamma)$ then $H_{\bm{\sigma}}$ here is precisely $G_{\bma}$ from Theorem \ref{thm:versal} with $\bma=-\bm{\sigma}$, so it does not depend on Theorems \ref{thm:versal} or \ref{thm:stability xi=0}, just on the calculations of Section\,\ref{sec:family of RE}.

For the stability, let $\bm{\mu}=(x,y,z)$ and $(\II-\II_r)^{-1} = \mathrm{diag}[a,b,c]$ with $a>b>c$.  Then with $\bm{\sigma}=0$ the $x$- and $z$-axes consist of stable relative equilibria, and the $y$-axis of linearly unstable relative equilibria, as for the ordinary free rigid body.  Note that here the $\TT^3$-reduced system is a phase space of dimension 6 so the symplectic slice at $\bm{\mu}=0$ reduces to 0 and we are in the situation of Remark\,\ref{rmk:stability}\,(b), so all elliptic \re\ are in fact Lyapunov stable. 

Now suppose we consider $\bm{\sigma}=(\sigma_1,0,0)$ with $\sigma_1\neq0$.  In other words we have `activated' the rotor along the principal direction of lowest moment of inertia (although looking at (\ref{eq:rotors Legendre}) shows the idea of `activation' is not entirely accurate)  This will give a family of relative equilibria as in Fig.\,\ref{fig:deformations}\,(ii) (and with all green curves being made red).  This means that if the satellite is given a small angular momentum, there are two relative equilibria, both rotating about the axis with the activated rotor, and both are stable.  As the angular momentum is increased, one of these (the one of lower energy) will undergo a supercritical pitchfork bifurcation---see Fig.\,\ref{fig:bifurcations}---so that the rotation about the axis becomes unstable, while there are two new relative equilibria, rotating about axes initially close to the given axis. As the angular momentum increases further, the unstable \re\ stabilizes again, and two new unstable \re\ appear. 

A similar scenario occurs if we activate the rotor along the axis of greatest moment of inertia, except here it is the \re\ with greater energy (for given angular momentum) that loses stability in a supercritical bifurcation, before stabilizing again.

If instead the rotor along the middle axis is activated, there are again two stable \re\ and as $\|\bm{\mu}\|$ is increased, they both lose stability in supercritical pitchfork bifurcations; which one occurs first depends on the relative values of $a,b,c$ (if $a-b=b-c$ then they occur simultaneously). 

The reader is invited to supply the storyline if two or three of the rotors are activated, following Fig.\,\ref{fig:deformations}. But in every case, it should be noted that \re\ with sufficiently small angular momentum $\bm{\mu}$ are always stable if $\bm{\sigma}\neq0$, which is the setting of Theorem\,\ref{thm:stability xi neq 0}.

Note that the energy-momentum discriminants in Figure\,\ref{fig:EM-discriminants} are in fact accurate for this system, as the Hamiltonian is of degree 2.

\subsection{A controlled version} 
Now suppose the rotors are used as control mechanisms, and their angular velocities relative to the body can be fixed.  That is, put $\dot\theta_i = u_i$, fixed (a constraint).  The Lagrangian is then
$$L=\half \bm{\omega}^T (\II-\II_r)\bm{\omega} + \half(\bm{\omega}+{\bm{u}})^T\II_r(\bm{\omega}+{\bm{u}}),$$
where $\uu\in\RR^3$ is constant. 
The corresponding Hamiltonian, with variables $A, \bm{\mu}$ is
$$H(A,\bm{\mu}) = \half\bm{\mu}^T\II^{-1}\bm{\mu} - \bm{\mu}^T\II^{-1}\bm{\alpha}
$$
where $\bm{\alpha} = \II_r\dot{\bm{\theta}}$ (a constant vector whose components are the angular momenta of the spinning rotors). 

The three components of $\bm{\alpha}$ give three coefficients in place of $\dot{\bm{\theta}}$, which unfold the singularity occurring when $\bm{\alpha}=0$, and provided the three principal moments of inertia of the rigid body are distinct, we obtain the same unfolding as described above, and again the Hamiltonian is of degree 2 so the energy-momentum discriminants shown in Fig.\,\ref{fig:EM-discriminants} are accurate.

%%%%%%%%%%%%%%%%%%%%%%%%%%%%%%%
\section{Singularity theory and deformations}
\label{sec:singularity theory}

In this section we use techniques from singularity theory to prove Theorem~\ref{thm:versal}.   Recall from Section~\ref{sec:reduction} that given an $\SO(3)$-invariant Hamiltonian $H$ on $\kP$ we define the reduced Hamiltonian $h:\so(3)^*\to\RR$, and the set $\RE\subset \so(3)^*$ of relative equilibria coincides with the set of critical points of the energy-Casimir map
$$(h,\jj) = \left(h(x,y,z),\,\half(x^2+y^2+z^2)\right).$$
Thus $\RE$ is the set where the rank of the Jacobian matrix,
\begin{equation} \label{eq:d(phi,h)}
F(x,y,z) = \d(\jj,h) = \left[\matrix{h_x& h_y& h_z\cr x&y&z}\right]
\end{equation}
is at most 1.  Let $V\subset \Mat(2,3)$ consist of all $2\times3$ matrices of rank at most 1.  Then $\RE=F^{-1}(V)$. {Here $h$ is a single (reduced) Hamiltonian, so corresponds to $h_0$ from earlier sections, and similarly $F$ corresponds to $F_0$ (ie, with $\bma=0$).}

Perturbations of the Hamiltonian $H$ produce perturbations of the Jacobian matrix $F$, and hence deformations of $\kR=F^{-1}(V)$.  Singularity theory provides a technique for deciding which deformations of $F^{-1}(V)$ arise by perturbing $F$, and the appropriate equivalence relation on $F$ is called $\KV$-equivalence and was introduced by Damon \cite{Damon-1987}, see also \cite{Damon-1991}. We recall this briefly before continuing with the proof.

Let $F,G:X\to Y$, and let $V\subset Y$ (everything in sight should be considered as germs).  Then $F$ and $G$ are said to be $\KV$-equivalent if there is a diffeomorphism $\psi$ of $X$ and a diffeomorphism $\Psi$ of $X\times Y$ preserving $X\times V$ and of the form $\Psi(x,y) = (\psi(x),\, \psi_1(x,y))$, such that
$$\Psi(x,\, F(x)) = (\psi(x),\, G(\psi(x)))\,;$$ 
that is, $\Psi$ maps the graph of $F$ to the graph of $G$. It follows in particular that $\psi(F^{-1}(V))=G^{-1}(V)$, so that these two sets are diffeomorphic.  If $V$ is just a point, then this reduces to ordinary $\mathcal{K}$-equivalence.

There are several rings and modules we need to consider. For $F:X\to Y$, let $\kE_X$ and $\kE_Y$ be the rings of germs at 0 of $C^\infty$ functions on $X$ and on $Y$ respectively.  Similarly, $\theta_X$ and $\theta_Y$ are the modules over $\kE_X$ and $\kE_Y$ of germs of vector fields on $X$ and $Y$ respectively. For $V\subset Y$ we write $\theta_V$ for the submodule of $\theta_Y$ consisting of vector fields tangent to $V$ (often denoted $\mathrm{Derlog}(V)$ in the singularity theory literature).  And finally one writes $\theta(F)$ for the $\kE_X$-module of `vector fields along $F$', meaning sections of the pull back of $TY$ to $X$ via $F$, or more prosaically if $X$ and $Y$ are linear spaces, $\theta(F)$ is the $\kE_X$ module of all germs at $0$ of maps $X\to Y$

In our setting, $X=\RR^3$, $Y=\Mat(2,3)$, and $V\subset Y$ is the set of matrices of rank at most 1. Now, there is a natural action of $\GL(3)\times\GL(2)$ on $\Mat(2,3)$ given by $(A,B)\cdot M = A M B^{-1}$, and of course this action preserves $V$; indeed, $V$ consists of just two orbits of this action: the origin and the set of matrices of rank equal to 1. The infinitesimal version of this action gives a map $\gl(2)\times\gl(3)\to \theta_Y$, whose image therefore lies in $\theta_V$.  Write $\theta_V'\subset\theta_V$ for the $\kE_Y$-module generated by these vector fields. (It seems likely that $\theta_V'=\theta_V$, though for the computations we will see that in fact $\theta_V'$ suffices.)

Now, $\dim(\gl(3)\times\gl(2))=13$, but the element $(I,-I)$ acts trivially, so that $\theta_V'$ has 12 generators; they are vector fields such as
$$\pmatrix{a_{11}&a_{12}&a_{13}\cr 0&0&0}, \quad \pmatrix{0&0&0 \cr a_{11} &a_{12}&a_{13}}, \quad \pmatrix{0&a_{11}&0\cr 0&a_{21}&0},\quad\dots
$$
Here the matrix $(u_{ij})$ refers to the vector field $\sum_{i,j} u_{ij}\frac{\partial}{\partial a_{ij}}$. In words, the generators are obtained by taking a single row $\pmatrix{a_{i1}&a_{i2}&a_{i3}}$ of $A$ and placing it in either row with 0s in the other row (there are 4 such vector fields), and then taking a single column $\pmatrix{a_{1j}\cr a_{2j}}$ and placing it in any column and filling the remaining 2 columns with zeros (9 such vector fields).

Given any $\kE_Y$-module $\theta$ of vector fields on $Y$, one defines two $\kK_\theta$ tangent spaces of a map $F:X\to Y$ to be the $\kE_X$-submodules of $\theta(F)$, the $\kK_\theta$-tangent space
$$T\,\kK_\theta\cdot F = tF(\maxid_X\theta_X) + F^*\theta,$$
and the extended $\kK_\theta$-tangent space
$$T\,\kK_{\theta,e}\cdot F = tF(\theta_X) + F^*\theta.$$
Here 
\begin{itemize}
\item $tF(\theta_X)$ means the $\kE_X$-module generated by the partial derivatives of $F$, and so  $tF(\maxid_X\theta_X)$ is the maximal ideal $\maxid_X$ times $tF(\theta_X)$, and
\item $F^*\theta=\kE_X\left\{v\circ F\mid v\in\theta\right\}$, the $\kE_X$-module generated by the vector fields in $\theta$ composed with $F$.
\end{itemize}
The ordinary tangent space is used for finite determinacy properties while the extended one is used for versal deformations.

Now consider the map $F$ defined in (\ref{eq:d(phi,h)}). The partial derivatives of $F$ lead to
$$tF(\theta_X) =\kE_X\left\{\pmatrix{h_{xx}&h_{xy}&h_{xz}\cr 1&0&0},\; \pmatrix{ h_{yx}&h_{yy}&h_{yz}\cr 0&1&0},\; \pmatrix{h_{zx}&h_{zy}&h_{zz}\cr 0&0&1} \right\},$$
where subscripts refer to partial derivatives. The second term in $T\kK_V\cdot F$ contains 12 generators such as
$$\pmatrix{h_x&h_y&h_z\cr 0&0&0},\quad \pmatrix{x&y&z\cr 0&0&0},\quad \pmatrix{0&0&h_y\cr 0&0&y}.$$

Now apply this to the function $h(x,y) = \half(ax^2+by^2+cz^2)$, with $a,b,c$ distinct. We obtain $T\kK_V\cdot F = \maxid_X\theta(F)$, and 
$$T\kK_{V,e}\cdot F = \maxid_X \theta(F) + \RR\left\{\pmatrix{1&0&0\cr 0&0&0},\;  \pmatrix{0&1&0\cr 0&0&0},\;\pmatrix{0&0&1\cr 0&0&0} \right\},
$$
where $\maxid_X$ is the maximal ideal of functions on $X=\RR^3$ that vanish at 0, so $\maxid_X=\left<x,\,y,\,z\right>$.

Here we have used $\theta_V'$ rather than $\theta_V$ and \emph{a priori} the expression above is for the corresponding module $T\kK_V'\cdot F$.  However, the fact that the vector fields tangent to $V$ all vanish at the origin in $\Mat(2,3)$ implies there are no other elements of $T\kK_V\cdot F$, so that for this function $h$ one has indeed that $T\kK_V'\cdot F=T\kK_V\cdot F$, and similarly for the extended tangent spaces.

We are now in a position to prove Theorem \ref{thm:versal}.

\begin{proofof}{Theorem \ref{thm:versal}}
Since $\kK_V$-equivalence is one of Damon's geometric subgroups of $\kK$, it follows that the usual finite determinacy and versal deformation theorems hold \cite{Damon-1984,Damon-1991}. In particular, with the family $\kG$ as in the statement of the theorem, one has that
$$\theta(F) = T\kK_{V,e}\cdot F + \RR\cdot\left\{\frac{\partial \kG}{\partial \alpha},\; \frac{\partial \kG}{\partial \beta},\; \frac{\partial \kG}{\partial \gamma}\right\}.$$
Consequently, by Damon's theorems, the family $\kG_{\bma}$ (with $\bma=(\alpha,\beta,\gamma)$) is a $\kK_V$-versal deformation of $F$, which is what is required for the theorem.

Furthermore, $\maxid_X\theta(F)\subset T\kK_V\cdot F$ implies that $F$ is 1-determined w.r.t.\ $\kK_V$-equivalence. It follows that if $h(x,y,z)$ has the same 2-jet as $g$ then the map $F$ associated to $\kG$ and $\kH$ have the same 1-jet so are $\kK_V$-equivalent; consequently $g$ is in this sense $2$-determined.
\end{proofof}

\begin{remark}
The vector field constructions of singularity theory will all produce diffeomorphisms whose linear part at the origin is the identity.  However, allowing more general diffeomorphisms, one can show further that $\kH$ is equivalent to the version of $\kG$ with coefficients $a=1,\,b=0,\,c=-1$ say. Indeed, one can be obtained from the other by row operations on the matrix in (\ref{eq:d(phi,h)}).
\end{remark}

\paragraph{Other singularities}
We considered above the $\KV$-equivalence arising from a quadratic Hamiltonian at the origin.  To understand the equivalence better, two further examples are worth considering:

\begin{itemize}
\item[(1)] If  $\d h(0,0,0)\neq0$, we have $h(x,y,z) = ax+by+cz+\cdots$ with $(a,b,c)\neq(0,0,0)$. In this case $T\KV\cdot F = \theta(F)$, so $h$ is `stable' in the appropriate sense: any sufficiently small deformation $h'$ of $h$ gives rise to a map $F'$ which is $\KV$-equivalent to $F$, so having diffeomorphic sets of relative equilibria (this is not surprising: we know it is just a non-singular curve through the origin).

\medskip

\item[(2)] At points away from $0$ in $\RR^3$, $\KV$-equivalence is more familiar: choose local coordinates so that $\jj(x,y,z) = z$ (this is possible in a neighbourhood of a point with $\d\jj\neq0$), then locally
$$F=\pmatrix{h_x&h_y&h_z\cr 0&0&1}.$$
Thus $F(x,y,z)\in V$ if and only if $h_x=h_y=0$. That is, this approach is finding critical points of $h$ as a function of $(x,\,y)$ with parameter $z$, and the $\kK_V$-equivalence of $F$ corresponds to $\kK$-equivalence of 1-parameter families of gradients of functions.  This is not the same as unfolding equivalence (or bifurcation-equivalence), as the relation does not distinguish $z$ as a parameter. For example the differentials of the maps $(h,\jj) = (x^2-y^2,z)$ and $(xz,z)$ are $\KV$-equivalent. 
\end{itemize}

%%%%%%%%%%%%%%%%%%%%%%%%%%%%%%%

\end{document}